\documentclass[oneside,openany,a4paper]{amsart}
\usepackage[a4paper, margin=1in]{geometry}
\usepackage{amssymb}
\usepackage[centertags]{amsmath}
\usepackage{amsthm}
\usepackage{amsfonts}
\usepackage{eucal}
\usepackage{mathrsfs}
\usepackage{enumerate}

\usepackage[all,cmtip]{xy}
\usepackage{graphicx, eepic}
\usepackage{dsfont, soul, stackrel, tikz-cd, graphicx, etoolbox, multirow}

\usepackage{lmodern}
\usepackage[utf8]{inputenc}
\usepackage[T1]{fontenc}
\usepackage{pdfrender, xcolor}

\usepackage{bbm}
\usepackage{pdfpages}
\usepackage{titletoc}
\usepackage{booktabs}
\usepackage[shortlabels]{enumitem}
\usepackage{mathtools}
\usepackage{thmtools}

\usepackage{epstopdf}
\usepackage{epsfig}
\usepackage{microtype}

\usepackage{hyperref}
\hypersetup{colorlinks}
\hypersetup{
    bookmarks=true,         
    unicode=false,          
    pdftoolbar=true,        
    pdfmenubar=true,        
    pdffitwindow=false,     
    pdfstartview={FitH},    
    pdftitle={My title},    
    pdfauthor={Author},     
    pdfsubject={Subject},   
    pdfcreator={Creator},   
    pdfproducer={Producer}, 
    pdfkeywords={keyword1} {key2} {key3}, 
    pdfnewwindow=true,      
    colorlinks=true,       
    linkcolor=red,          
    citecolor=magenta,        
    filecolor=violet,      
    urlcolor=blue      
}
\usepackage{multirow, url}
\usepackage{textcomp}
\DeclareMathAlphabet{\cmcal}{OMS}{cmsy}{m}{n}
\newtheoremstyle{thm}
  {3pt}
  {3pt}
  {\em}
  {0pt}
  {\bfseries}
  {}
  {5pt}
  {}
\newtheoremstyle{rem}
  {3pt}
  {3pt}
  {}
  {0pt}
  {\bfseries}
  {.}
  {5pt}
  {}

\newtheorem{thm}{Theorem}[section]
\newtheorem{cor}[thm]{Corollary}
\newtheorem{lem}[thm]{Lemma}
\newtheorem{prop}[thm]{Proposition}

\newtheorem{conj}[thm]{Conjecture}

\theoremstyle{definition}
\newtheorem{defn}[thm]{Definition}

\theoremstyle{rem}
\newtheorem{rem}[thm]{{Remark}}

\numberwithin{equation}{section} \numberwithin{table}{section}
	
\newtheorem*{thm*}{Theorem}
\newtheorem*{rem*}{Remark}
\newtheorem*{rems*}{Remarks}
\newtheorem*{exam*}{Example}
\newtheorem*{exams*}{Examples}

\makeatletter
\newcommand{\neutralize}[1]{\expandafter\let\csname c@#1\endcsname\count@}
\makeatother

\usepackage{enumerate}

\def\bos#1{{\mathbf{#1}}}

 \newcommand{\G}{GL_2^{+}(\mathbb{Q})}

 \newcommand{\pa}[1]{\frac{\partial}{\partial{#1}}}
  \newcommand{\prt}[2]{\frac{\partial{#1}}{\partial{#2}}}

  \newcommand{\pr}{\operatorname{prim}}
  \newcommand{\frC}{\mathfrak{C}}

  \newcommand{\Des}{\operatorname{\mathfrak{Des}}}

  \newcommand{\nc}{\newcommand}
  \newcommand{\be}{\begin{eqnarray*}}
  \newcommand{\ee}{\end{eqnarray*}}
  \newcommand{\bea}{\begin{eqnarray}}
  \newcommand{\eea}{\end{eqnarray}}

   \nc{\bei}{\begin{itemize}}
   \nc{\eei}{\end{itemize}}
   \nc{\bee}{\begin{enumerate}}
   \nc{\eee}{\end{enumerate}}
   \nc{\bet}{\begin{thm}}
   \nc{\eet}{\end{thm}}
   \nc{\bed}{\begin{defn}}
   \nc{\eed}{\end{defn}}
   \nc{\bel}{\begin{lem}}
   \nc{\eel}{\end{lem}}
   \nc{\bep}{\begin{prop}}
   \nc{\eep}{\end{prop}}
   \nc{\bec}{\begin{corollary}}
   \nc{\eec}{\end{corollary}}
   \nc{\ber}{\begin{rem}}
   \nc{\eer}{\end{rem}}
   \nc{\beex}{\begin{example}}
   \nc{\eeex}{\end{example}}

   \nc{\bpm}{\begin{pmatrix}}
   \nc{\epm}{\end{pmatrix}}
   \nc{\bspm}{\left(\begin{smallmatrix}}
   \nc{\espm}{\end{smallmatrix}\right)}

\newcommand{\cA}{\mathcal{A}}
\newcommand{\cB}{\mathcal{B}}
\newcommand{\cC}{\mathcal{C}}

\newcommand{\cE}{\mathcal{E}}

\newcommand{\cH}{\mathcal{H}}

\newcommand{\cK}{\mathcal{K}}

\newcommand{\cO}{\mathcal{O}}

\newcommand{\bA}{\mathbb{A}}

\newcommand{\bC}{\mathbb{C}}

\newcommand{\bQ}{\mathbb{Q}}

\newcommand{\bZ}{\mathbb{Z}}

\newcommand{\BP}{\mathbf{P}}

\nc{\frf}{\mathfrak{f}} 

\newcommand{\frg}{\mathfrak g}

\nc{\frs}{\mathfrak{s}}  
\nc{\frt}{\mathfrak{t}} 
\nc{\fru}{\mathfrak{u}}
  
\nc{\lsl}{\mathfrak{sl}}
\nc{\lgl}{\mathfrak{gl}}

\nc{\upsi}{\underline{\psi}}
\nc{\uchi}{\underline{\chi}}

\DeclareMathOperator{\GL}{GL}

\DeclareMathOperator{\Res}{Res}

\DeclareMathOperator{\End}{End}

\newcommand{\lra}{\longrightarrow}    

\nc{\surjto}{\twoheadrightarrow}
\nc{\ts}{\times}
\nc{\ds}{\displaystyle}
\nc{\nd}{\noindent}  
\nc{\ud}{\underline}
\nc{\ov}{\overline}
\nc{\maplra}[1]{\buildrel #1 \over \lra}
\nc{\mapto}[1]{\buildrel #1 \over \to}
\nc{\setb}[1]{\{  #1\}}

 \nc{\cHom}{\mathcal{H}om}

\nc{\cdruur}[8] {\begin{CD} 
#1 @>#2>> #3\\ 
@AA#4A @AA#5A\\ 
#6 @>#7>> #8 
\end{CD} }
\nc{\cdrddr}[8] {\begin{CD} 
#1 @>#2>> #3\\ 
@VV#4V @VV#5V\\ 
#6 @>#7>> #8 
\end{CD} }

\nc{\dia}[8]{\xymatrix{ 
&#1 \ar@{-}[ld]_{#2} \ar@{-}[rd]^{#3} \\
#4 \ar@{-}[rd]_{#6} & &#5 \ar@{-}[ld]^{#7}\\ 
&#8} }

\nc{\diam}[9]{\xymatrix{ 
&#1 \ar@{-}[ld]_{#2}  \ar@{-}[d]^{#3} \ar@{-}[rd]^{#4} \\
#5 \ar@{-}[rd]_{#8}     & #6 \ar@{-}[d]_{#9}      & #7   \ar@{-}[ld]^{2} \\
& \bQ} } 

\nc{\sumn}[2][n]{#2_{1} +#2_{2}+ \cdots + #2_{#1}}
\nc{\poly}[3][n]{#2_{#1}#3^{#1} +#2_{#1-1}#3^{#1-1}  \cdots + #2_{1} #3+ #2_0}
\nc{\dpoly}[3][n]{#1#2_{#1}#3^{#1-1} +(#1-1)#2_{#1-1}#3^{#1-1}  \cdots +2 #2_{2} #3+ #2_1}
\nc{\mpoly}[3][n]{#3^{#1} +#2_{#1-1}#3^{#1-1}  \cdots + #2_{1} #3+ #2_0}

\nc{\vpar}[4]{    \left \{ \begin{array}{cc} #1 & \textrm{if } #2, \\
&\\
#3 & \textrm{if } #4. 
\end{array}\right. }

\nc{\ary}[5]{#1: \left\{ \begin{array}{ll} #2 &\mapsto #3 \\ #4 &\mapsto #5 \end{array} \right.}
 \nc{\bedm}{\begin{displaymath}}
 \nc{\eedm}{\end{displaymath}}
 \nc{\art}{\hbox{\bf Art}^\Z}
 \nc{\bvx}{\bos{B\!\!V}_{\! \!\ud G}}

\newcommand{\pmat}{\left(\begin{matrix}}   
\newcommand{\epmat}{\end{matrix}\right)}   
\newcommand{\psmat}{\left(\begin{smallmatrix}}    
\newcommand{\epsmat}{\end{smallmatrix}\right)}
\nc{\twotwo}[4]{\pmat #1 & #2 \\ #3 & #4 \epmat}
\nc{\thrthr}[9]{\pmat #1 & #2 & #3 \\ #4 & #5 & #6 \\ #7 & #8 & #9 \epmat}
\nc{\stwotwo}[4]{\psmat #1 & #2 \\ #3 & #4 \epsmat}
\nc{\sthrthr}[9]{\psmat #1 & #2 & #3 \\ #4 & #5 & #6 \\ #7 & #8 & #9 \epsmat}

\def\eqalign#1{\null\,\vcenter{\openup\jot\m@th
\ialign{\strut\hfil$\displaystyle{##}$&$\displaystyle{{}##}$\hfil
\crcr#1\crcr}}\,}

\def\eqn#1#2{
\xdef #1{(\nsecsym\the\meqno)}
\global\advance\meqno by1
$$#2\eqno#1\eqlabeL#1
$$}

\def\a{\alpha}
\def\b{\beta}

\def\d{\delta}  
\def\e{\varepsilon} 
  
\def\g{\gamma}  \def\G{\Gamma}

\def\l{\lambda}  
\def\m{\mu}
\def\n{\nu}
\def\r{\rho}

\def\o{\omega}  \def\O{\Omega}

\def\t{\tau}

\def\C{\mathbb{C}}

\def\Z{\mathbb{Z}}

\def\Q{\mathbb{Q}}

\def\mC{\mathfrak{C}}

\def\mL{\mathfrak{L}}

\def\rd{\partial}

\def\Ker{\hbox{Ker}\;}
\def\Im{\hbox{Im}\;}

\def\mod{\hbox{ }mod\hbox{ }}

\def\wt{\hbox{\it wt}}
\def\ch{\hbox{\it ch}}

\catcode`\@=11 

\global\newcount\nsecno \global\nsecno=0
\global\newcount\meqno \global\meqno=1
\def\newsec#1{\global\advance\nsecno by1
\eqnres@t
\section{#1}}
\def\eqnres@t{\xdef\nsecsym{\the\nsecno.}\global\meqno=1}
\def\sequentialequations{\def\eqnres@t{\bigbreak}}\xdef\nsecsym{}

\def\draftmode{\message{ DRAFTMODE }

{\count255=\time\divide\count255 by 60 \xdef\hourmin{\number\count255}
\multiply\count255 by-60\advance\count255 by\time
\xdef\hourmin{\hourmin:\ifnum\count255<10 0\fi\the\count255}}}
\def\nolabels{\def\wrlabeL##1{}\def\eqlabeL##1{}\def\reflabeL##1{}}
\def\writelabels{\def\wrlabeL##1{\leavevmode\vadjust{\rlap{\smash%
{\line{{\escapechar=` \hfill\rlap{\tt\hskip.03in\string##1}}}}}}}%
\def\eqlabeL##1{{\escapechar-1\rlap{\tt\hskip.05in\string##1}}}%
\def\reflabeL##1{\noexpand\llap{\noexpand\sevenrm\string\string\string##1}
}}

\nolabels

\def\eqn#1#2{
\xdef #1{(\nsecsym\the\meqno)}
\global\advance\meqno by1
$$#2\eqno#1\eqlabeL#1
$$}

\def\eqalign#1{\null\,\vcenter{\openup\jot\m@th
\ialign{\strut\hfil$\displaystyle{##}$&$\displaystyle{{}##}$\hfil
\crcr#1\crcr}}\,}

\vspace{.2in}

   \nc{\hr}{[\![\hbar]\!]}

\catcode`\@=12 

\def\Fr#1#2{{#1\over#2}}

\nc{\bt}{\mathbf{t}}
\draftmode

\begin{document}


\title[polynomial realizations and deformation of period integrals]{Polynomial realizations of period matrices of projective smooth complete intersections and their deformation}

\author{Yesule Kim}
\email{yesule@snu.ac.kr}
\address{Department of Mathematical Sciences, Seoul National University GwanAkRo 1, Gwanak-Gu, Seoul 08826, South Korea. }
\author{Jeehoon Park}
\email{jeehoonpark@postech.ac.kr}
\address{Department of Mathematics, POSTECH (Pohang University of Science and Technology), San 31, Hyoja-Dong, Nam-Gu, Pohang, Gyeongbuk, 790-784, South Korea. }
\author{Junyeong Park}
\email{junyeongp@gmail.com}
\address{Department of Mathematics, POSTECH (Pohang University of Science and Technology), San 31, Hyoja-Dong, Nam-Gu, Pohang, Gyeongbuk, 790-784, South Korea. }


\begin{abstract}
Let $X$ be a smooth complete intersection over $\bC$ of dimension $n-k$ in the projective space $\BP^n_{\bC}$, for given positive integers $n$ and $k$. 
For a given integral homology cycle $[\g] \in H_{n-k}(X(\bC),\bZ)$, the period integral is defined to be a linear map from the de Rham cohomology group to $\bC$ given by $[\o] \mapsto \int_\g \o$. The goal of this article is to interpret this period integral as a linear map from the polynomial ring with $n+k+1$ variables to $\bC$ and use this interpretation to develop a deformation theory of period integrals of $X$. The period matrix is an invariant defined by the period integrals of the \textit{rational} de Rham cohomology, which compares the \textit{rational} structures ($\bQ$-subspace structures) of the de Rham cohomology over $\bC$ and the singular homology with coefficient $\bC$.
As a main result, when $X'$ is another projective smooth complete intersection variety deformed from $X$, we provide an explicit formula for the period matrix of $X'$ in terms of the period matrix of $X$ and the Bell polynomials evaluated at the deformation data.
Our result can be thought of as a modern deformation theoretic treatment of the period integrals based on the Maurer-Cartan equation of a dgla (differential graded Lie algebra).


\end{abstract}

\maketitle
\tableofcontents


\section{Introduction} \label{sec1}

The goal of this article is to give a modern deformation theoretic interpretation for the period integrals and period matrices of smooth projective complete intersection variety $X$ over the complex field $\bC$ via a dgla (differential graded Lie algebra) and a DGBV (differential Gerstenhaber-Batalin-Vilkovisky) algebra whose construction are based on a polynomial quotient description of the primitive middle-dimensional algebraic de Rham cohomology of $X$.
Let $n$ and $k$ be positive integers such that $n \geq k \geq 1$. 
Let $X_{\ud G}$ be a smooth complete intersection variety over $\bC$ of dimension $n-k$ embedded in the projective space $\BP^{n}=\BP^n_{\bC}$ over $\bC$.
We use $\underline x = [x_0: x_1: \cdots: x_n]$ as a projective coordinate of the projective $n$-space $\BP^n$ and let $G_1(\ud x), \cdots, G_k(\ud x)$ be defining homogeneous polynomials in $\bC[\ud x]$ such that $\deg(G_i)=d_i$ for $i = 1, \cdots, k$.
 We introduce more variables $y_1, \cdots, y_k$ corresponding to $G_1, \cdots, G_k$. Let $N=n+k+1$ and 
\begin{align}\label{fao}
A :=\bC[\ud q]= {\bC}[q_\m]_{\m=1,\cdots, N}
\end{align}
where $q_1=y_1, \cdots, q_k=y_k$ and $q_{k+1}=x_0, \cdots, q_N=x_{n}$.
We consider the Dwork potential
\begin{align}\label{dpot}
S(\ud q) := \sum_{\ell=1}^k y_\ell \cdot G_{\ell}(\underline{x}).
\end{align}
Then we will construct a linear map
\begin{eqnarray}\label{Jmap}
J=J_{\ud G}:A \to H^{n-k}_{dR}(X_{\ud G}; \bC)
\end{eqnarray}
 where  $H^{n-k}_{dR}(X_{\ud G}; \bC)$ is the middle-dimensional algebraic de Rham cohomology group of $X_{\ud G}/\bC$ (see subsection \ref{sec3.2}),\footnote{This map is due to Griffiths in the hypersurface case, \cite{Gr69},  Terasoma in the equal degree complete intersection case, \cite{Ter90}, and Konno in the general case, \cite{Ko91}.}
whose kernel is the subspace
\begin{eqnarray}\label{kker}
\cK_{\ud G}:=\bigoplus_{i=1}^N \left( \pa{q_i}+\prt{ S(\ud q) }  {q_i}\right) A  
\end{eqnarray}
and the image is isomorphic to the primitive middle-dimensional algebraic de Rham cohomology group $H^{n-k}_{dR, \pr}(X_{\ud G}; \bC) \subset H^{n-k}_{dR}(X_{\ud G};\bC)$. In other words, we have an induced isomorphism $\bar J_{\ud G} : A/\cK_{\ud G} \xrightarrow{\simeq} H^{n-k}_{dR, \pr}(X_{\ud G}; \bC)$.

For any homology class $[\g] \in H_{n-k}(X_{\ud G}(\C), \bZ)$, we consider the following period integral from the de Rham cohomology to $\bC$:
\begin{eqnarray}\label{ytr}
C^{\ud G}_{\g}: H^{n-k}_{dR}(X_{\ud G};\bC) \longrightarrow \C,\qquad [\varpi]\mapsto  \int_\g \varpi,
\end{eqnarray}
where $\varpi$ is a representative of a cohomology class $[\varpi]\in H^{n-k}_{dR}(X_{\ud G}; \bC)$. Then 
\begin{eqnarray}\label{polyreal}
\cC^{\ud G}_{\g}:=C^{\ud G}_{\g} \circ J_{\ud G} : A \to \bC
\end{eqnarray}
is a linear map from the polynomial ring $A$ to $\bC$, which we call \emph{a polynomial realization} of the period integral $\int_{\g}$. We will use the simplified notation $\cC_{\g}$ instead of $\cC^{\ud G}_{\g}$, if the context is clear.

We explain our main theorems regarding deformation formulas of period integrals and period matrices using the polynomial realization.
For this, let $X_{\ud U}\subset \BP^n$ be a smooth projective complete intersection over $\bC$ which is deformed
from $X=X_{\ud G}$ by homogeneous polynomials $\ud H = (H_1(\ud x),$ $ \cdots, 
H_k(\ud x))$, i.e. $\ud U =(U_1(\ud x), $ $\cdots, U_k(\ud x))$ with $U_i(\ud x)=G_i(\ud x) + H_i(\ud x)$ for each $i =1, \cdots, k$ are the defining equations
for $X_{\ud U}$.
Note that smooth projective complete intersections with fixed degrees $d_1, \cdots, d_k$ have same topological types and their singular (vanishing) homologies and singular (primitive) cohomologies are isomorphic (Corollary \ref{CI-diffeo-type}).
To make the context more clear, we make an assumption that our homology class $[\gamma]$ is supported in $X_{\ud G}(\bC)\cap X_{\ud U}(\bC)\subset \BP^n(\bC)$. In this case, we use Proposition \ref{CI-ambient-isotopy} to get a particular identification
\begin{align}\label{homology-identification}
\xymatrix{H_\bullet(X_{\ud G}(\bC),\mathbb{Z}) \ar[r]^-\sim & H_\bullet(X_{\ud U}(\bC),\mathbb{Z}) & \text{such that } \quad [\gamma] \ar@{|->}[r] & [\gamma]}
\end{align}
which we fix for the rest of the article. Using this identification, we can directly compare $\cC_\g^{\ud G}$ and $\cC_\g^{\ud U}$. The following theorem is motivated by Proposition \ref{klemma} (a modern deformation theoretic interpretation of period integrals) and the proof relies on Dimca's result on cohomology of complete intersections, \cite{Dimca95}.

\begin{thm} \label{firthm}
Let $\Gamma=\sum_{i=1}^k y_i H_i(\ud x)$. 
For any homogeneous polynomial $u \in A$, the power series\footnote{Here $e^\G=1 + \G + \frac{\G \cdot \G}{2} + \cdots $ and, for $x \in A$, we think of
$\cC_\g^{\ud G}(x \cdot e^\G)=\cC_\g^{\ud G}(x) + \cC_\g^{\ud G}(x\cdot \G) + \cC_\g^{\ud G}(x\cdot \frac{\G \cdot \G}{2})+\cdots $ as a formal expression.}
$$
\cC_\g^{\ud G} \left(u e^\Gamma\right):=\sum_{n=0}^{\infty}\cC_\g^{\ud G} \left(u \frac{\Gamma^n}{n !}\right)
$$
converges to the polynomial realization $\cC_\g^{\ud U}(u)$ of the period integral of $X_{\ud U}$.
\end{thm}

The next goal is to use the deformation formula in Theorem \ref{firthm} to express $\cC_\g^{\ud U}(u)$ as 
a homotopy Lie formula (also called, an $L_\infty$-homotopy formula) with the Bell polynomials evaluated at the deformation data.
For any linear map $\cC:A\to \bC$, we define a multi-linear map $\phi_m^{\cC}: S^m(A) \to \bC$ for each $m \geq 1$, where $S^m(A)$ is the symmetric $m$-th tensor product of $A$, as follows:
define $\phi_1^{\cC}= {\cC}$ and 
\be
\phi_m^{\cC}(x_1, \cdots, x_m) = \phi_{m-1}^{\cC} (x_1, \cdots, x_{m-2}, x_{m-1}\cdot x_m) -
\sum_{\substack{\pi \in P(m), |\pi|=2 \\ m-1 \nsim_\pi m}} \phi^{\cC} (x_{B_1})\cdot \phi^{\cC}(x_{B_2}), \quad m \geq 2,
\ee
where $P(m)$ is the set of partitions of $\{1, \cdots, m\}$ and we refer to Definition \ref{partition} for details.
We will explain later (subsection \ref{sec3.3}) that $\ud \phi^{\cC}=\phi^{\cC}_1, \phi^{\cC}_2, \cdots $ is an example of an $L_\infty$- morphism between certain $L_\infty$-algebras.
The \emph{Bell polynomials} $B_n(x_1,\cdots,x_n)$ are defined by the power series expansion
\begin{align}\label{Bpoly}
\exp\left(\sum_{i\geq1}x_i\frac{t^i}{i!}\right)=1+\sum_{n\geq1}B_n(x_1,\cdots,x_n)\frac{t^n}{n!}.
\end{align}

\begin{thm}\label{sthm}
Let $\Gamma=\sum_{i=1}^k y_i H_i(\ud x)$. For any homogeneous polynomial $u \in A$, we have
\begin{align*}
\cC_\g^{\ud U}(u)=
\cC_\g^{\ud G}(u)+\sum_{m\geq1}\sum_{\substack{j+k=m \\ j,k\geq0}}\frac{1}{j!k!} B_j\left(\phi^{\cC_\g^{\ud G}}_1(\Gamma),\cdots,\phi^{\cC_\g^{\ud G}}_j(\Gamma, \cdots, \Gamma)\right)\cdot\phi^{\cC_\g^{\ud G}}_{k+1}(\Gamma, \cdots, \Gamma,u)
\end{align*}
where $B_0=1$ and $\phi^{\cC_\g^{\ud G}}_1(\Gamma,u):=\phi^{\cC_\g^{\ud G}}_1(u)=\cC_\g^{\ud G}(u)$.
\end{thm}

The last goal is to provide a deformation formula of period matrices of $X_{\ud U}$ and $X_{\ud G}$ using Theorem \ref{firthm}. 
The period matrix of $X_{\ud G}$ is about a comparison between a $\bQ$-rational (or $\bZ$-integral) structure on $H_{dR}^{n-k}(X_{\ud G};\bC)$ and  a $\bQ$-rational (or $\bZ$-integral) structure on $H_{n-k}(X_{\ud G}(\bC), \bC)$. 
Sections 6 and 7, \cite{Gr69} are wonderful places to read about the period matrices of smooth projective hypersurfaces.
For a brief review of period matrices, we need to tell more about the map $J_{\ud G}: A \to H_{dR}^{n-k}(X_{\ud G};\bC)$. 
We introduce an additive grading on $A$:
\begin{eqnarray} \label{weight}
\wt(y_i)=1, \quad \text{for } i=1, \cdots, k, \quad \wt(x_{j}) = 0, \quad \text{for }  j=0, \cdots, n. 
\end{eqnarray}
Let $A_{(w)}$ be the subspace of $A$ consisting of elements whose $\wt$ is $w$.
Consider the Hodge decomposition of $H^{n-k}_{dR}(X_{\ud G};\bC) $:
$$
H^{n-k}_{dR}(X_{\ud G};\bC) \simeq \bigoplus_{p+q=n-k} H_{dR}^{p,q}.
$$
Then it is known (Proposition \ref{mthm}  \textit{(b)}) that $J_{\ud G}$ sends $\bigoplus_{0\leq q \leq w} A_{(q)}/\cK_{\ud G} $ to $\bigoplus_{0\leq q \leq w} H_{dR}^{p,q}$ for every $w \leq n-k$.
We choose polynomials $e_1,\cdots, e_{\d_0}, e_{\d_0+1},\cdots, 
 e_{\d_1},\cdots,e_{\d_{n-(k-1)}+1},$ $\cdots,e_{\d_{n-k}}$
 in $\bQ[\ud q]$ such that $\{ e_i \mod  \cK_{\ud G} : i=1, \cdots, \d_{n-k}\}$ is a $\bQ$-basis of $A/\cK_{\ud G}=\bQ[\ud q]/\cK_{\ud G} $, the set
$\{J_{\ud G}(e_1),\cdots, J_{\ud G} (e_{\d_0})\}$ gives a basis for the subspace $H_{dR}^{n-k, 0}$,
and the set $\{J_{\ud G}(e_{\d_{q-1} +1}),\cdots, J_{\ud G}(e_{\d_{q}})\}$, $1\leq q \leq n-k$, gives a basis for the subspace 
$H_{dR}^{n-k- q, q}$. 

Let $\cH$ be the image of $\bQ[\ud q]$ under the map $J_{\ud G}$ so that $\cH$ is a $\bQ$-vector subspace of $H^{n-k}_{dR,\pr}(X; \bC)$ satisfying $\cH \otimes_\bQ \bC \simeq H^{n-k}_{dR,\pr}(X; \bC)$.
Thus $\cH$ is an example of a $\bQ$-rational structure on $H^{n-k}_{dR,\pr}(X; \bC)$.
We use the $\bQ$-basis  
$\{\e_\a=J_{\ud G}(e_\a)\}_{\a \in I}$ of $\cH$ where $I=I_{0}\sqcup I_{1}\sqcup\cdots\sqcup I_{n-k}$ with the notation
$\{\e^{j}_{\a}\}_{\a\in I_j}=\{ \e_{\d_{j-1} +1},\cdots, \e_{\d_{j}}\}$.

Let $\{\g^{\ud G}_\a\}_{\a\in I}$ (respectively, $\{\g^{\ud U}_\a \}_{\a\in I}$) be a $\bZ$-basis of $H_{n-k}(X_{\ud G}(\bC), \bZ)_0$, the homology group of vanishing cycles of $X_{\ud G}(\bC)$ (respectively, $H_{n-k}(X_{\ud U}(\bC), \bZ)_0$). Then the isomorphism (\ref{homology-identification}) restricts to
\begin{align*}
\xymatrix{\iota: H_{n-k}(X_{\ud G}(\bC),\bZ)_0 \ar[r]^-\sim & H_{n-k}(X_{\ud U}(\bC), \bZ)_0}.
\end{align*}
Moreover,
$$
\dim_\bC H_{n-k}(X_{\ud G}(\bC),\bC)_0 = \dim_\bC H^{n-k}_{dR,\pr}(X_{\ud G};\bC) =|I|
$$
and there is a unique matrix $B=(B_{\b}^{\a})\in \GL_{|I|}(\bZ)$ corresponding to $\iota$ such that
\begin{eqnarray}\label{ibc}
\sum_{\a \in I} \g^{\ud G}_\a B_{\b}^{\a} = \g^{\ud U}_\b, \quad \b \in I.
\end{eqnarray}

We define the period matrix $\Omega(X_{\ud G})$ of $X_{\ud G}$ with respect to $\{e_\a \mod \cK_{\ud G}\}_{\a \in I}$ and $\{\g^{\ud G}_\a\}_{\a \in I}$:
\begin{align}\label{pma}
\Omega(X_{\ud G})_\b^\a := \int_{\g^{\ud G}_\a} J_{\ud G}(e_\b),
\quad  \a, \b \in I.
\end{align}
Let $I'$ be the set of indices $i$ such that $H_i(\ud x) \neq 0$.
We assume that the cardinality of $I$ is bigger than $|I'|$, and 
view $I'$ as a subset of $I$ (allowing a slight abuse of notation).
The intersection pairing 
\begin{align*}
H_{n-k}(X_{\ud G}(\bC), \bZ)_0 \otimes H_{n-k}(X_{\ud G}(\bC), \bZ)_0 \to \bZ.
\end{align*}
and the basis $\{\g^{\ud G}_\a\}_{\a\in I}$ gives the intersection matrix  $Q[{\g^{\ud G}_\a}]$. 
The period matrix \eqref{pma} is determined up to a substitution $\O(X_{\ud G})
\to M \cdot  \O(X_{\ud G}) \cdot \Lambda$ where $M \in \GL_{|I|}(\bQ)$ is a base change matrix for $\cH$ which preserves the increasing Hodge filtration (this has a lower triangular form: see (6.2), \cite{Gr69}) and $\Lambda \in \GL_{|I|}(\bZ)$ is a base change matrix\footnote{For more refined conditions on $\Lambda$, see sections 6 (odd dimensional hypersurfaces - a symplectic matrix) and 7 (even dimensional hypersurfaces - an orthogonal matrix), \cite{Gr69}.  One can put a similar condition on $\Lambda$ for $k > 1$ depending on the parity of the dimension $n-k$.}  for $H_{n-k}(X_{\ud G}(\bC), \bZ)_0$ satisfying $\Lambda \cdot Q[{\g^{\ud G}_\a}] \cdot \Lambda^T = Q[{\g^{\ud G}_\a}]$.

\begin{thm}\label{tthm}
If we assume \eqref{aind}, then
we can construct
 $\{ u_i \mod  \cK_{\ud U} : i=1, \cdots, \d_{n-k}\}$, which is a $\bQ$-basis of $\bQ[\ud q]/\cK_{\ud U}$,
and a power series $T^\rho(\ud t) \in \bQ[[ \ud t]]$ for each $\r\in I$ (where $\ud t=(t^\a)_{\a \in I}$ are formal variables)
 such that
\eqn\drilike{
\Omega(X_{\ud {U}}) = D \cdot \Omega(X_{\ud G}) \cdot B
}
where 
$$
\Omega(X_{\ud U})_\b^\a := \int_{\g^{\ud U}_\a} J_{\ud U}(u_\b),
\quad  \a, \b \in I,
$$
$D$ is the $|I|\times |I|$ matrix whose $(\b, \r)$-entry is given by
$$
D_\b^\r =  \left( \frac{\rd}{\rd t^\b} T^{\rho}(\ud t)\right)
\Big|_{\substack{t^\a=1, \a \in I'\\ t^\a=0,  \a\in I\setminus I'}}, \quad \b, \r \in I,
$$
and $B$ is given in \eqref{ibc}.
\end{thm}

This theorem says that $\Omega(X_{\ud G})$ and $\Omega(X_{\ud U})$ are ``transcendental'' invariants but their relationship is ``algorithmically computable up to desired precision (the coefficients of higher powers of $t^\a$ in $\frac{\rd}{\rd t^\b} T^{\rho}(\ud t)$'': if we know the period matrix $\Omega(X)$ and the polynomials $G_{\ell}(\ud x), H_{\ell}(\ud x), \ell = 1, \cdots, k$, then there is an algebraic algorithm to compute (approximately) the period matrix $\Omega(X_{\ud U})$. 

\begin{rem}\label{ccc}
(a)
Theorem \ref{tthm} has an arithmetic application when $k=1$: in \cite{KKP17}, the first two authors and Kwang-Hyun Kim provided an explicit
algorithm to compute the inverse value of the modular $j$-function based on the case of elliptic curves of Theorem \ref{tthm}.

(b)
Theorem \ref{tthm} generalizes and strengthens the results for smooth projective hypersurfaces in \cite{PP16} to smooth projective complete intersections.

(c)
Theorem \ref{firthm} under \eqref{ibc} can be restated as follows: $ \sum_{\a \in I}{B^\a_\d}\cdot \cC_{\g^{\ud G}_\a} \left(u\cdot e^\Gamma\right)$\text{ converges to } $\cC_{\g^{\ud U}_\d} (u)$ for $\d \in I$ and $u \in A$.
\end{rem}

Now we briefly explain the contents of each section.
Section \ref{sec2} is about how to find a DGBV algebra $\bvx$ for $H_{dR,\pr}^{n-k}(X_{\ud G};\bC)$.
In subsection \ref{sec2.1}, we explain the geometric idea behind our homotopy Lie theory.
Subsection \ref{sec2.2} is devoted to associate a Lie algebra representation $\rho$ to $X_{\ud G}$
and construct a model for $\bvx$, a dual Chevalley-Eilenberg complex for $\rho$.
In subsection \ref{sec2.4}, our construction of the DGBV algebra $\bvx$ attached to $X_{\ud G}$ is given. 

In section \ref{sec3}, we use the homotopy Lie theory in section \ref{sec2} to 
study deformations of period integrals and matrices of smooth projective complete intersections; we prove three main theorems.
In subsection \ref{sec3.1}, we explain the polynomial realization of the period integral. Then we give a proof of Theorem \ref{firthm} in subsection \ref{sec3.2}.
In subsection \ref{sec3.3}, we briefly review the formal deformation theory and the descendant functor, which is needed for Theorem \ref{sthm}.
Subsection \ref{sec3.4} (respectively, subsection \ref{sec3.5}) is devoted to the proof of Theorem \ref{sthm} (respectively, subsection \ref{tthm}).

Finally, we include Appendix \ref{apa} for homology cycles of projective smooth complete intersections.

%

\subsection{Acknowledgement}
The work of Jeehoon Park was partially supported by BRL (Basic Research Lab) through the NRF (National Research Foundation) of South Korea (NRF-2018R1A4A1023590). 
The work of YK was partially supported by Basic Science Research Program through the National Research Foundation of Korea (NRF-2019R1C1C1008614)
and was partially supported by BK21 PLUS SNU Mathematical Sciences Division.
The authors would like to thank Jae-Suk Park for helping them initiating this project and providing useful comments and warm support. Jeehoon Park thanks KIAS (Korea Institute for Advanced Study), where the part of work was done, for its hospitality.

\section{A dgla associated to a complete intersection and its cohomology} \label{sec2}

In this section we explain how to associate a dgla (differential graded Lie algebra) to smooth projective complete intersection $X_{\ud G}$, whose $0$-th cohomology is isomorphic to the primitive middle dimensional cohomology of $X_{\ud G}$. In fact, we will do slightly more: We construct a DGBV (differential Gerstenhaber-Batalin-Vilkovisky) algebra which has more structures (in particular, the super-commutative algebra structure) than a dgla. This DGBV algebra will play a key role in the polynomial realizations of the period integrals and their deformation theory.

\subsection{Geometric idea behind the homotopy Lie theory} \label{sec2.1}
If we are interested in the cohomology of the smooth projective complete intersection variety $X_{\ud G}$ of dimension $n-k$, then the primitive middle dimensional cohomology
$H^{n-k}_{dR,\pr}(X_{\ud G};\bC)$ is the most interesting piece because the other degree cohomologies and non-primitive pieces can be easily described in terms of the cohomology of the projective space $\BP^n$ due to the weak Lefschetz theorem and the Poincare duality. 
For the computation of $H^{n-k}_{dR,\pr}(X_{\ud G};\bC)$, the Gysin sequence and ``the Cayley trick'' play important roles. There is a long exact sequence, called the Gysin sequence:
\begin{eqnarray*}
\cdots \to H^{n+k-1}_{dR}(\BP^n ; \bC) \to H^{n+k-1}_{dR}(\BP^n \setminus X_{\ud G};\bC) \xrightarrow{\Res_{\ud G}} H^{n-k}_{dR}(X_{\ud G};\bC) \xrightarrow{\operatorname{Gys}} H^{n+k}_{dR}(\BP^n;\bC)  \to \cdots
\end{eqnarray*}
where $\Res_{\ud G}$ is the residue map (see p.\,96 of \cite{Dim95}) and $\operatorname{Gys}$ is the Gysin map. This sequence gives rise to an isomorphism 
\begin{eqnarray}\label{residuemap}
\Res_{\ud G}:H^{n+k-1}_{dR}(\BP^n \setminus X_{\ud G}; \bC) \xrightarrow{\sim} H^{n-k}_{dR, \operatorname{prim}}(X_{\ud G};\bC).
\end{eqnarray}
The Cayley trick is about translating a computation of the cohomology of the complement of a complete intersection into a computation of the cohomology of the complement of a hypersurface in a bigger space. Let $\cE=\cO_{\BP^n}(d_1) \oplus \cdots \oplus \cO_{\BP^n}(d_k)$ be the locally free sheaf
of $\cO_{\BP^n}$-modules with rank $k$. Let $\BP(\cE)$ be the projective bundle associated to $\cE$ with fiber $\BP^{k-1}$ over $\BP^n$. Then $\BP(\cE)$ is the smooth projective toric variety with Picard group isomorphic to $\bZ^2$ whose (toric) homogeneous coordinate ring is given by
\begin{eqnarray} \label{fa}
A:=A_{\BP(\cE)}:=\bC[y_1, \cdots, y_k, x_0, \cdots, x_n]
\end{eqnarray}
where $y_1, \cdots, y_k$ are new variables corresponding to $G_1, \cdots, G_k$. This is the polynomial ring introduced in \eqref{fao}.
There are two additive gradings $\ch$ and $\wt$, called the charge and the weight\footnote{This grading was already introduced in \eqref{weight}}, corresponding to the Picard group $\bZ^2$: 
$$
\ch(y_i)=-d_i, \quad \text{for } i=1, \cdots, k, \quad \ch(x_{j}) = 1, \quad \text{for }  j=0, \cdots, n,
$$ 
$$
\wt(y_i)=1, \quad \text{for } i=1, \cdots, k, \quad \wt(x_{j}) = 0, \quad \text{for }  j=0, \cdots, n. 
$$ 
Then 
$$S(\ud y, \ud x):= \sum_{j=1}^k y_j G_j(\ud x) \in {A} $$
defines a hypersurface $X_S$ in $\BP(\cE)$.
The natural projection map $pr_1:\BP(\cE) \to \BP^n$ induces a morphism 
$
\BP(\cE)\setminus X_S \to \BP^n \setminus X_{\ud G}
$
which can be checked to be a homotopy equivalence (the fibers are affine spaces). Hence there exists an isomorphism 
$$
H^{n+k-1}_{dR}(\BP(\cE) \setminus X_S; \bC) \xrightarrow{s^*} H^{n+k-1}_{dR}(\BP^n \setminus X_{\ud G}; \bC)
$$
where $s$ is a section to $pr_1$.
The cohomology group $H^{n+k-1}_{dR}(\BP(\cE) \setminus X_S; \bC)$ of a hypersurface complement in $\BP(\cE)$ can be described explicitly in terms of the de-Rham cohomology of $\BP(\cE)$ with poles along $X_S$.
Based on this, one can further show that there is an isomorphism $\bar\varphi_S$
\begin{eqnarray}
 \bar\varphi_S: {A}/\cK_{\ud G} \xrightarrow{\sim} H^{n+k-1}_{dR}(\BP(\cE) \setminus X_S; \bC)
\end{eqnarray}
where we recall that 
$\cK_{\ud G}$ (see \eqref{kker}) is the sum of the images of the endomorphisms $\pa{y_i}+\prt{ S }{y_i}, \pa{x_j}+\prt{S}{x_j}$ of ${A}$ ($i =1, \cdots, k, j=0, \cdots, n$).
 The linear isomorphism $\bar J_{\ud G}: {A}/\cK_{\ud G}  \to H^{n-k}_{dR,\pr}(X; \bC)$ in the introduction is defined by the following sequence of above isomorphisms:
 {\small{
 \begin{eqnarray}\label{gma}
\qquad \quad \bar J_{\ud G}: {A}/\cK_{\ud G}  \xrightarrow{\bar \varphi_S}   H^{n+k-1}_{dR}(\BP(\cE) \setminus X_S; \bC) \xrightarrow{s^*}  H^{n+k-1}_{dR}(\BP^n \setminus X_{\ud G}; \bC)
\xrightarrow{\Res_{\ud G}} H^{n-k}_{dR,\pr}(X_{\ud G}; \bC).
\end{eqnarray}
 }}
 The concrete description of the first map $\bar\varphi_S$ in \eqref{gma} is crucial for our deformation theory; we refer to subsection \ref{sec3.2} for its concrete description. 
 \begin{rem}
 A more familiar explicit description of the primitive cohomology $H^{n-k}_{dR,\pr}(X_{\ud G}; \bC)$ is in terms of the Jacobian ideal of $S$.\footnote{The Jacobian ideal $Jac(S)$ is given by the sum of the images of the endomorphisms $\prt{ S }{y_i}, \prt{S}{x_j}$ of ${A}$ ($i =1, \cdots, k, j=0, \cdots, n$).} We refer to \cite[Theorem 1]{Dim95} or \cite{Ko91} and see also \cite{Gr69} for the pioneering work of Griffiths in the case $k=1$, the smooth projective hypersurface case; in our notations, there exists an isomorphism $\bar\Phi$
$$
\bar \Phi: A_{c_{\ud G}}/Jac(S) \cap A_{c_{\ud G}} \xrightarrow{\sim} H^{n-k}_{dR,\pr}(X_{\ud G}; \bC)
$$
where sub-index $c_{\ud G}$ means the submodule in which the charge is 
$$
c_{\ud G} := \sum_{i=1}^k d_i - (n+1).
$$
\end{rem}
\begin{rem} \label{fourrem}
One could try to use a lift $\Phi:A_{c_{\ud G}} \to  H^{n-k}_{dR,\pr}(X_{\ud G}; \bC)$ of $\bar\Phi$ to obtain a polynomial realization of $\int_\g$ in \eqref{ytr}. But that is not a good choice for several reasons: 

(1) The map $J_{\ud G}$ can be given in the level of cochain complexes (see subsection \ref{sec3.2}) but the map $\Phi$ can not be given in the level of cochain complexes.\footnote{One can directly show that there is a non-canonical isomorphism between $A_{c_{\ud G}}/Jac(S) \cap A_{c_{\ud G}}$ and $A/\cK_{\ud G}$, i.e. their dimensions are the same. But this isomorphism is not induced from a linear map between $A$ and $A_{c_{\ud G}}$.} 

(2) $A_{c_{\ud G}}$ does not have commutative algebra structure unless $c_{\ud G}=0$ (Calabi-Yau case) induced from polynomial multiplication of $A=\bC[\ud q]$. The algebra structure of $A$ plays a crucial role in Theorem \ref{firthm} and \ref{sthm}. 

(3) The kernel $\cK_G$ of $J_{\ud G}$ enables us to understand $\int_\g$ as a BV integral\footnote{A BV (Batalin-Vilkovisky) integral is one of natural frameworks to study the Feynman path integral in physics.}, i.e. the polynomial realization \eqref{polyreal} lifts to a morphism from the DGBV algebra $\bvx$ (see subsection \ref{sec2.4}).
\end{rem}

%

 \subsection{Lie algebra representation attached to projective smooth complete intersections} \label{sec2.2}
The $\bC$-vector space $A/\cK_{\ud G}$ in \eqref{gma}
leads us to 
consider the following Lie algebra representation. Let $q_1=y_1, \cdots, q_k=y_k$ and $q_{k+1}=x_0, \cdots, q_{n+k+1}=x_{n}$. Let $\frg_\bC$ be an abelian Lie algebra over $\bC$ of dimension $n+k+1$. Let $u_{1}, u_2, \cdots, u_{n+k+1}$ be a ${\bC}$-basis of $\frg_\bC$. 
We associate a Lie algebra representation $\rho$ on ${A}$ of 
$\frg_\bC$ as follows:
\begin{eqnarray*}
\rho(u_i) := \pa{q_i}+\prt{ S(\ud q) }  {q_i}, \text { for $i=1, 2, \cdots, n+k+1$}.
\end{eqnarray*}
We extend this ${\bC}$-linearly to
get a Lie algebra representation $\rho: \frg_\bC\to\End_{{\bC}}({A})$. Then the 0-th Lie algebra homology is isomorphic to 
${A}/\cK_{\ud G}$. Also, the $(n+k+1)$-th Lie algebra cohomology is isomorphic to ${A}/\cK_{\ud G}$. In fact, the Chevalley-Eilenberg cohomology complex is the twisted de-Rham complex $(\Omega_{\bA^{n+k-1}}^\bullet, d + dS)$ of the affine space $\bA^{n+k-1}$. On the other hand, one can consider the Chevalley-Eilenberg homology complex.
More precisely, we use the cochain complex $(\cA_\rho^\bullet, K_\rho)$, which we call \textit{the dual Chevalley-Eilenberg complex}, such that $H^i(\cA_\rho^\bullet, K_\rho)$ $ \simeq H_{-i}(\frg_\bC, {A})$ for $i \in \Z$:
\be
\cA^\bullet=\cA_{\rho}^\bullet&=& {A}[\eta_{1},\eta_2, \cdots, \eta_N]=\bC[\ud q][\eta_1,\eta_2 \cdots, \eta_N], \\
K_{\ud G}=K_\rho&=&  \sum_{i=1}^N \left(\prt{ S(\ud q)}{q_i} + \pa{q_i}\right) \pa{\eta_i}:\cA_{\rho} \to \cA_{\rho}.
\ee
 We have
$$
\xymatrix{0 \ar[r] & \cA_\rho^{-(n+k+1)} \ar[r]^-{K_\rho} & \cA_\rho^{-(n+k)} \ar[r]^-{K_\rho} & \cdots \ar[r]^-{K_\rho} & \cA_\rho^{-1} \ar[r]^-{K_\rho} & \cA_\rho^{0}={A} \ar[r] & 0}
$$
where
$$
 \cA_\rho^{-s} = \bigoplus_{1\leq i_1 < \cdots < i_s\leq n+k+1} {A} \cdot \eta_{i_1} \cdots \eta_{i_s}, \quad 0 \leq s \leq n+k+1.
$$
Since one easily sees that $H^{n+k-1-s}(\Omega_{\bA^{n+k-1}}^\bullet, d + dS) \xrightarrow{\sim}
H^s(\cA_\rho^\bullet, K_\rho)$ for $s \in \bZ$ (using the fact that $\frg_\bC$ is abelian), i.e. their differential module structures are isomorphic, either complexes can be used to study the primitive middle-dimensional cohomology of $X$. 
The twisted de-Rham complex, sometimes called the algebraic Dwork complex, seems to be used more often in the literature (for example, \cite{Dim95}).
But note that \textit{the multiplication structure} (the $\bZ$-graded commutative wedge product structure) on $(\Omega_{\bA^{n+k-1}}^\bullet, d + dS)$ is different from the one (the $\bZ$-graded commutative algebra structure with the rule $\eta_i \eta_j = -\eta_j \eta_i$) on $(\cA, K_{\ud G})=(\cA_\rho, K_\rho)$. We observe that the product structure on $(\cA, K_{\ud G})$ induces a \textit{homotopy Lie structure} and a DGBV structure on $(\cA, K_{\ud G})$ (by measuring the failure of $K_{\ud G}$ being a derivation of the product successively), which governs a deformation theory of cochain complexes and maps. This type of algebraic structures was studied systematically in \cite{PP16} under the name``the descendant functor'' by 
the second-named author with his collaborator.

We also consider the following differentials on $\cA=\cA_{\ud G}$:
\begin{eqnarray*}
Q_{\ud G}&:=&\sum_{i=1}^N \prt{S(\underline q)}{q_i} \pa{\eta_i}: \cA_{\ud G} \to \cA_{\ud G}, \\
\Delta &:=& K_{\ud G}-Q_{\ud G}=\sum_{i=1}^N  \pa{q_i} \pa{\eta_i}:\cA_{\ud G} \to \cA_{\ud G}.
\end{eqnarray*}
Both $Q_{\ud G}$ and $K$ also have degree 1 and $\Delta^2=Q_{\ud G}^2=0$. Therefore we have
\be
\Delta  Q_{\ud G} +Q_{\ud G}  \Delta =0.
\ee
Note that
$$
Jac(S)=Q_{\ud G} (\cA^{-1}) \quad \text{and} \quad \cK_{\ud G} = K_{\ud G} (\cA^{-1}).
$$

\subsection{A DGBV algebra}\label{sec2.4}
We briefly recall definitions of G-algebras, BV-algebras, GBV-algebras, and DGBV-algebras.
\begin{defn}\label{bvd}
Let $k$ be a field. Let $(\cC,\cdot)$ be a unital $\Z$-graded super-commutative and associative $k$-algebra.
Let $[\bullet, \bullet]: \cC \otimes \cC \to \cC$ be a bilinear map of degree 1.

(a) $(\cC, \cdot, [\cdot, \cdot])$ is called a G-algebra (Gerstenhaber algebra) over $k$ if
\begin{align*}
  \quad [a, b] &= (-1)^{|a||b|}[b,a],\\
  [a, [b,c]\!]&= (-1)^{|a|+1}[\![a,b],c]+(-1)^{(|a|+1)(|b|+1)} [b, [a,c]\!],\\
 \quad [a,b \cdot c]&= [a, b] \cdot c +(-1)^{(|a|+1)\cdot |b|} b \cdot [a,c],
 \end{align*}
 for any homogeneous elements $a, b, c \in \cC.$

(b) $(\cC,\cdot, {L},\ell_2^{L})$ is called a GBV(Gerstenhaber-Batalin-Vilkovisky)-algebra\footnote{$(\cC,\cdot, {L})$ is called a BV algebra, if $(\cC,\cdot, {L},\ell_2^{L})$ is a GBV algebra.} over $k$ where
\begin{align}\label{elltwo}
\ell_2^{{L}}(a,b):= {L}(a \cdot b)-{L}(a)\cdot b -(-1)^{|a|} a\cdot {L}(b), \quad a,b \in \cC,
\end{align}
 if 
$(\cC, {L}, \ell_2^{L})$ is a shifted\footnote{In the usual definition of dgla, the differential ${L}$ has degree 1 and the Lie bracket $\ell_2^{L}$ has degree 0. In our case, both ${L}$ and $\ell_2^{L}$ have degree 1.} dgla (differential graded Lie algebra) and
$(\cC, \cdot, \ell_2^{L})$ is a G-algebra.

 (c) $(\cC, \cdot, {L}, \ell_2^{L}, {P})$, where ${P}:\cC \to \cC$ is a linear map of degree 1, is called a DGBV(differential Gerstenhaber-Batalin-Vilkovisky) algebra if $(\cC, \cdot, {L},  \ell_2^{L}(\cdot, \cdot))$ is a GBV algebra
 and $(\cC,\cdot,{P})$ is a CDGA(commutative differential graded algebra), i.e.
 $$
 {P}^2=0, \quad {P}(a \cdot b) = {P}(a) \cdot b + (-1)^{|a|} a \cdot {P}(b), \quad a,b \in \cC,
 $$
 such that $(L+P)^2=0$.
\end{defn}

The DGBV algebra plays an important role to state the (closed string) mirror symmetry conjecture: for example, see \cite{BK} for its precise statement and a complex analytic construction of $B$-side DGBV algebra
attached to a compact Calabi-Yau complex manifold. Also, see \cite{KKP} for a relationship between our DGBV algebra of Proposition \ref{mthm} and the Barannikov-Kontsevich DGBV algebra in \cite{BK}.

Using the weight, we consider the increasing filtration $F^{\bullet}_{\wt}\cA$ defined by
\begin{eqnarray}\label{wf}
F^{j}_{\wt}\cA =\bigoplus_{0\leq w\leq j}\cA_{(w)}
\end{eqnarray}
where $\cA_{(w)}$ is the ${\C}$-subspace of $\cA$ generated by the monomials of weight $w$.
\begin{prop}
The weight filtration induces
a filtered  cochain  complex $(F^\bullet_{\wt}\cA, K_{\ud G})$.
\end{prop}
\begin{proof}
Note that $K_{\ud G}=Q_{\ud G}+\Delta$. Since $Q_{\ud G}$ preserves $\wt$ and $\Delta$ decreases $\wt$ by 1, the result follows.%
 \end{proof}

 \begin{prop} \label{mthm}
The quadruple $\bvx:=$ $ (\cA, \cdot, Q_{\ud G}, K_{\ud G}, \ell_2^{K_{\ud G}})$ associated to $X$ is a finitely generated DGBV algebra over $\bC$
with the following properties:

(a) We have a decomposition $\cA =\bigoplus_{-(n+k+1) \leq m \leq 0} \cA^m$ and there is a canonical isomorphism 
$$
J=J_{\ud G}: H_{K_{\ud G}}^0 (\cA) \mapto{\sim} H^{n-k}_{dR, \pr}(X_{\ud G}; {\bC})
$$
where $H_{K_{\ud G}}^0 (\cA)=\cA^{0}/K_{\ud G}(\cA^{-1})$ is the $0$-th cohomology module.

(b) The increasing weight filtration on $(\cA, K_{\ud G})$ in \eqref{wf} goes to the increasing Hodge filtration on $H_{dR, \pr}^{n-k}(X_{\ud G};\bC)$ under $J_{\ud G}$.

\end{prop}
\begin{proof}
 The fact that the quadruple $\bvx:=$ $ (\cA, \cdot, Q_{\ud G}, K_{\ud G}, \ell_2^{K_{\ud G}})$ is a DGBV algebra  can be deduced from straightforward (but lengthy) computations. 
  In fact, one can check that the dual Chevalley-Eilenberg complex attached to a representation of any abelian Lie algebra consisting of differential operators of order $\leq 1$ becomes a DGBV algebra for fairly general reasons: we refer to sections 2.2 and 2.3, \cite{PP16} for details.\footnote{This fact is intimately related to the BV-BRST formalism in the math-physics literature.}

\textit{(a)} follows from the isomorphism \eqref{gma} by noting that
$$
\cA^{0}=A, \quad K_{\ud G}(\cA^{-1}) =\cK_{\ud G}.
$$

\textit{(b)} is a reinterpretation of \cite[Theorem 1]{Dim95}.
\end{proof}

\section{Deformation formula for the period integrals of $X$} \label{sec3}

\subsection{Polynomial realizations of the period integrals}\label{sec3.1}

\begin{defn}\label{dp}(the polynomial realization of $C_\g^{\ud G}$ given in \eqref{ytr})
For any $[\g] \in H_{n-k}(X_{\ud G}(\bC), \bZ)$, we define a $\C$-linear map $\cC_\g^{\ud G}: \cA \to \C$ by 
$$
\cC_\g^{\ud G} (f) :=\vpar{ \int_\g J_{\ud G} (f) }{f \in \cA^0=A}
{0}{\text{otherwise}}
$$
where $J=J_{\ud G}$ is given in \eqref{gma}.
\end{defn}

\begin{prop}\label{kkill}
The polynomial realization map $\cC_\g^{\ud G}$ is a cochain map from $(\cA, K_{\ud G})$ to $(\C,0)$.
\end{prop}

\begin{proof}
Since $J_{\ud G} \circ K_{\ud G} =0$ by the construction, it follows that $\cC_\g^{\ud G} \circ K_{\ud G}=0$.
 \end{proof}

Note that $\cA =A [\ud \eta]$, where $\eta_\m \eta_\n= -\eta_\n \eta_\m$, and define $\ch(\eta_\m)$ and $\wt(\eta_\m)$ by
$$
\ch(q_\m)+\ch(\eta_\m)=0,\quad
\wt(q_\m)+\wt(\eta_\m)=1.
$$
We have a decomposition
$$
\cA = \bigoplus_{-N\leq j \leq 0}\bigoplus_{\l \in \Z}\bigoplus_{w \geq 0}\cA^j_{\l, (w)},
$$
where $u \in \cA^j_{\l, (w)}$ if and only if $\ch(u)=\l$, $\wt(u)=w$ and the cohomological grading of $u$ is $j$.
For example, $S(\ud{q})\in \cA^0_{0,(1)}$ (the Dwork potential $S(\ud{q})$ was defined in \ref{dpot}).
Since $K_{\ud G}$ preserves the charge grading, we have a cochain subcomplex $(\cA_\l, K_{\ud G})$ for each charge $\l \in \Z$ .
It follows that the cohomology $H=H(\cA,K_{\ud G})$ of the cochain complex $(\cA, K_{\ud G})$ also has a charge decomposition
$$
H =\oplus_{\l \in \Z} H_\l, \qquad H_\l := H(\cA_\l,  K_{\ud G}).
$$

\begin{prop}  \label{chc}
The cohomology $H =H(\cA, K_{\ud G})$ is concentrated in charge $c_{\ud G}$, i.e.
$H = H_{c_{\ud G}}$.
\end{prop}

\begin{proof}
Associated with the vector field $E_{\ch} :=\sum_{\m=1}^{N} \ch(q_\m)q_\m\Fr{\rd}{\rd q_\m}+\sum_{\m=1}^{N} \ch(\eta_\m)\eta_\m\Fr{\rd}{\rd \eta_\m}$ we define
$$
R:= \sum_\m \ch(q_\m) q_\m \eta_\m \in \cA^{-1}_{0}.
$$
Note that, for any $f\in \cA$,
$$
\d_R f:=\ell_2^{K_{\ud G}}(R, f) =  \sum_\m \ch(q_\m)\left(  q_\m\Fr{\rd}{\rd q_\m}- \eta_\m \Fr{\rd }{\rd \eta_\m}\right)f
$$
Hence for any $f \in \cA_\l$ we have $\d_R f = \l f$.
Note also that
$Q_{\ud G} (R) =0$, while
$$
K_{\ud G} (R) =- c_{\ud G}, \qquad c_{\ud G}=\sum_{\ell=1}^k d_\ell -n-1=\hbox{the background charge}.
$$
From the computation
$$
\eqalign{
K_{\ud G}(f\cdot R) 
&=\ell_2^{K_{\ud G}}(R, f) + K_{\ud G}(R)\cdot f + K_{\ud G}(f)\cdot R
\cr
&=\d_R f - c_{\ud G} f + K_{\ud G} (f)\cdot R
\cr
&= (\l -c_{\ud G})f + K_{\ud G}(f)\cdot R \quad (\text{if} \ f \in \cA_\l)
}
$$
we see that any $f \in \cA_\l\cap \Ker K_{\ud G}$  belongs to $\Im K_{\ud G}$ unless $\l=c_{\ud G}$, since
$ (\l -c_{\ud G})f= K_{\ud G}(f\cdot R)$. 
\end{proof}

This proposition and Proposition \ref{kkill} imply the following corollary.
\begin{cor}
The polynomial realization $\cC_\g^{\ud G}(f)$ is zero unless $f \in \cA^{0}_{c_{\ud G}}=A_{c_{\ud G}}$.
\end{cor}

\subsection{An explicit description of the map $J$ and proof of Theorem \ref{firthm} }\label{sec3.2}
Here we give an explicit description of the map $J=J_{\ud G}$ in \eqref{Jmap} and prove Theorem \ref{firthm}. In order to prove Theorem \ref{firthm}, we will need explicit descriptions of the maps $\bar\varphi_S$ and $s^*$ in \eqref{gma}.

 We will achieve this by using the algebraic (twisted) de Rham complexes and Dimca's result in \cite{Dim95}.
For any commutative $\bQ$-algebra $C$, let us consider the de Rham complex $(\Omega^\bullet_C, d)$ and the twisted de Rham complex $(\Omega^\bullet_C, d+ dS \wedge)$ for any element $S\in C$. We also define the charge and the weight on the de Rham complex by
\[
	\begin{cases}
		\ch(q_i)=\ch(dq_i), &i=1,2,\dots,N\\
		\wt(q_i)=\wt(dq_i). &i=1,2,\dots,N\\
	\end{cases}
\]
which provides us the weight and the charge decomposition of $\Omega$ such that
\[
	\Omega = \bigoplus_{-N\leq j \leq 0}\bigoplus_{\l \in \Z}\bigoplus_{w \geq 0}\Omega^j_{\l, (w)}.
\]


\begin{defn} 
We define a sequence of maps as follows (here $N=n+k+1$):
	\begin{enumerate}[(i)]
		\item Let $\mu:(\bigoplus_{r=-N+1}^0\mathcal{A}_{c_{\ud G}}^r,Q_{\ud G}+\Delta)\to\left(\bigoplus_{r=1}^{N}\left(\Omega_A^r\right)_0[-N],dS+d\right)$ be the cochain map defined by
		\[
			\mu(\underline{q}^{\underline{u}}\eta_{i_1}\cdots\eta_{i_\ell})=(-1)^{i_1+\cdots+i_\ell+\ell+1}\underline{q}^{\underline{u}}dq_1\cdots\widehat{dq_{i_1}}\cdots\widehat{dq_{i_\ell}}\cdots dq_N
		\]
		where $\underline{q}^{\underline{u}}= q_1^{u_1} \cdots q_N^{u_N}$.
		\item Let $\rho:\left(\bigoplus_{r=1}^{N}\left(\Omega_A^r\right)_0,dS+d\right)\to\left(\bigoplus_{r=1}^{N}\left(\Omega_{A[S^{-1}]}^r\right)_{0,(0)},d\right)$ be the cochain map defined by
		\[
			\rho(\underline{x}^{\underline{u}}\underline{y}^{\underline{v}}d \underline{x}_{\underline{\alpha}}\wedge d\underline{y}_{\underline{\beta}})=(-1)^{|\underline{v}|+|\underline{\beta}|-1}(|\underline{v}|+|\underline{\beta}|-1)!\frac{\underline{x}^{\underline{u}}\underline{y}^{\underline{v}}}{S^{|\underline{v}|}}d \underline{x}_{\underline{\alpha}}\wedge\frac{d \underline{y}_{\underline{\beta}}}{S^{|\underline{\beta}|}},
		\]
		where $\underline{x}^{\underline{u}} = x_0^{u_0} \cdots x_n^{u_n}$, $\underline{y}^{\underline{v}}=y_1^{v_1}\cdots y_k^{v_k}$, $d \underline{x}_{\underline{\alpha}}=dx_{\a_1}\wedge \cdots \wedge dx_{\a_i}$, and $d\underline{y}_{\underline{\beta}}=dy_{\b_1} \wedge \cdots \wedge dy_{\b_j}$ with  $|\ud v| =v_{1} + \cdots + v_{k}$ and $|\underline{\beta}|=j$.
		\item Let $\theta_{\ch}$ be the contraction operator with the vector field $\sum_{i=1}^N\ch(q_i)q_i\frac{\partial}{\partial q_i}$, and $\theta_{\wt}$ be the contraction operator with the vector field $\sum_{i=1}^N\wt(q_i)q_i\frac{\partial}{\partial q_i}$.
	\end{enumerate}
\end{defn}
\begin{defn}
Define $\varphi_S$ to be the composition of the following maps:
\[
	\begin{tikzcd}
		\mathcal{A}_{c_{\ud G}}^0 \arrow[r,"\mu"] & \left(\Omega^N_{A}\right)_0 \arrow[r,"\rho"] & \left(\Omega^N_{A[S^{-1}]}\right)_{0,(0)} \arrow[r,"\theta_{wt}\circ\theta_{ch}"] & \Omega_{B}^{N-2} \arrow[r,"\textrm{quotient}"] & \Omega_B^{N-2}/d(\Omega_B^{N-3})=H_{\mathrm{dR}}^{N-2}(\mathbf{P}(\mathcal{E})\setminus X_S;\bC)
	\end{tikzcd}
\]
where $B:=\mathbb{C}[\underline{q},S^{-1}]_{0,(0)}$ and $\mathbf{P}(\mathcal{E})\setminus X_S$ is a smooth affine variety whose coordinate ring is given by $B$. Then we extend the domain of $\varphi_S$ to $A=\cA^0$ by setting $\varphi_S(u)=0$ for $\ch(u) \neq c_{\ud G}$.
\end{defn}

From the explicit description of $\varphi_S$ and Proposition \ref{chc}, the following proposition follows.
\begin{prop}\label{kg}
The kernel of $\varphi_S:\cA^0 \to H_{\mathrm{dR}}^{n+k-1}(\mathbf{P}(\mathcal{E})\setminus X_S;\bC)$ is $\cK_{\ud G}=K_{\ud G} (\cA^{-1})= (Q_{\ud G}+\Delta)(\cA^{-1})$.
\end{prop}

By this proposition \ref{kg}, $\varphi_S$ induces a map $\bar\varphi_S: \cA^0/\cK_{\ud G} \xrightarrow{\sim} H_{\mathrm{dR}}^{n+k-1}(\mathbf{P}(\mathcal{E})\setminus X_S;\bC)$, which we define to be the first map in \eqref{gma}.

%
%
%

%

For an explicit description of a section $s$ of the second map $s^*$ in \eqref{gma}, we review the toric quotient construction of the projective bundle $\BP(\cE)$. Let $G=\C^*\times \C^*$. Let $U_G=(\C^{n+1}\setminus \{0\}) \times (\C^{k}\setminus \{0\})$, consider the following $G$-action on $U_G$:
\[
(u,v) (\ud x,\ud y)=(u x_0, \cdots, ux_n, u^{-d_1} vy_1, \cdots, u^{-d_k} v y_k).
\]
The hypersurface $X_S$ in $U_G$ is $G$-invariant. 
The lemma 17, \cite{Dim95}, says that the geometric quotient $U_G/G$ is naturally identified to the projective bundle $\BP(\cE)$.
 Let $d=lcm(d_1, \cdots, d_k)$ and define the positive integers $e_i=d/d_i$. According to page 98, \cite{Dim95}, one has a well-defined section
\[
s_{\ud G}(\ud x) =(\ud x, G_1^{e_1-1}(\ud x)\overline G_1^{e_1}(\ud x), \cdots, G_1^{e_k-1}(\ud x)\overline G_1^{e_k}(\ud x) )
\]
to the projection map $pr_1:\BP(\cE)\setminus X_S \to \BP^n \setminus X_{\ud G}$.
In summary, we have
$$
\bar J_{\ud G}= \Res_{\ud G} \circ s_{\ud G}^* \circ \bar \varphi_S.
$$

\begin{rem}
Our idea of proving Theorem \ref{firthm} is to use the Laplace transform
\begin{eqnarray}\label{laplace}
 \int_{0}^\infty \cdots \int_{0}^\infty  \ud y^{\ud i} u(\ud x) e^{S(\ud q)} dy_1 \cdots dy_k  
= \frac{(-1)^p i_1 ! \cdots i_k !  u(\ud x)  }{ G_1^{i_1+1} \cdots G_k^{i_k+1}}.
\end{eqnarray}
But the denominator of the RHS of \eqref{laplace} is not directly related to the denominators of the rational differential forms of neither $H_{\mathrm{dR}}^{n+k-1}(\mathbf{P}(\mathcal{E})\setminus X_S;\bC)$ nor $H^{n+k-1}_{dR}(\BP^n\setminus X_{\ud G};\bC)$. It rather relates to the denominator of the rational differential form of $H^n_{dR}(\BP^n \setminus D;\bC)$
where $D$ is the divisor defined by $G_1(\ud x) \cdots G_k(\ud x) =0$. This leads to necessity of finding a different description (Propositions \ref{pone} and \ref{dc}) of the map $s_{\ud G}^* \circ \bar \varphi_S$ through $H^n_{dR}(\BP^n \setminus D;\bC)$ and \cite[Proposition 10]{Dim95} turns out to give us the desired answer.
\end{rem}

Let
$$
\Omega_{\ud x} = \sum_{i=0}^n (-1)^i x_i (dx_0\wedge \cdots \wedge \hat {d x_i} \wedge \cdots \wedge dx_n), \quad 
\Omega_{\ud y} = \sum_{i=1}^k (-1)^i y_i (dy_1\wedge \cdots \wedge \hat {d y_i} \wedge \cdots \wedge dy_k).
$$
For a multi-index $\ud i=(i_1, \cdots, i_k)$ with $|\ud i|=i_1 + \cdots +i_k$, define the following rational differential form (\cite[p 93]{Dim95})
\[
\b(S, \ud y^{\ud i} u(\ud x)):=(-1)^{|\ud i|+1}\frac{i_1! \cdots i_k!\cdot u(\ud x)}{G_1^{i_1+1}(\ud x) \cdots G_k^{i_k+1}(\ud x)} \O_{\ud x},
\]
for $\ud y^{\ud i} u(\ud x) \in \cA_{c_{\ud G}}^0$. 
Also consider the rational differential form (\cite[p 90]{Dim95})
\[
\a(S,\ud y^{\ud i}u(\ud x) ):=(-1)^{n(k-1)+|\ud i|}(|\ud i|+k-1)!\frac{\ud y^{\ud i}u(\ud x) }{S(\ud q)^{k+|\ud i|}} \O_{\ud x} \O_{\ud y}
\]
where $\ud y^{\ud i}= y_1^{i_1} \cdots y_k^{i_k}$ and $\ud y^{\ud i}u(\ud x) \in \cA_{c_{\ud G}}^0$. 
\begin{prop} \label{pone}
For $\ud y^{\ud i}u(\ud x) \in \cA_{c_{\ud G}}^0$, we have
\begin{eqnarray}\label{result}
\bar \varphi_S ([\ud y^{\ud i}u(\ud x) ])= [\a(S,\ud y^{\ud i}u(\ud x) )]
\end{eqnarray}
where $[\cdot]$ means the cohomology class.
\end{prop}

\begin{proof}
\begin{align*}
\ud y^{\ud i}u(\ud x)&\xymatrix{ \ar@{|->}[r]^-\mu & }-y^iu(x)dy_1\wedge\cdots\wedge dy_k\wedge dx_0\wedge\cdots\wedge dx_n\\
&\xymatrix{ \ar@{|->}[r]^-\rho & }(-1)^{|i|+k}(|i|+k-1)!\frac{\ud y^{\ud i} u(\ud x)}{S^{|i|+k}}dy_1\wedge\cdots\wedge dy_k\wedge dx_0\wedge\cdots\wedge dx_n
\end{align*}
Now we apply $\theta_{wt}\theta_{ch}$ to $dy_1\wedge\cdots\wedge dy_k\wedge dx_0\wedge\cdots\wedge dx_n$:
\begin{align*}
&\quad\theta_{wt}\theta_{ch}(dy_1\wedge\cdots\wedge dy_k\wedge dx_0\wedge\cdots\wedge dx_n)\\
&=-\theta_{ch}\theta_{wt}(dy_1\wedge\cdots\wedge dy_k\wedge dx_0\wedge\cdots\wedge dx_n)\\
&=-\theta_{ch}(\Omega_{\ud y}\wedge dx_0\wedge\cdots\wedge dx_n)\\
&=-(\theta_{ch}\Omega_{\ud y})\wedge dx_0\wedge\cdots\wedge dx_n+(-1)^k\Omega_{\ud y}\wedge\Omega_{\ud x}\\
&=(-1)^{k+n(k-1)}\Omega_{\ud x}\wedge\Omega_{\ud y}-(\theta_{ch}\Omega_{\ud y})\wedge dx_0\wedge\cdots\wedge dx_n
\end{align*}
By using the definition of $\varphi_S$ and $|\ud i|=i_1 + \cdots +i_k$, we get
\begin{align*}
\varphi_S(\ud y^{\ud i}u(\ud x))&=\left[(-1)^{|\ud i|+n(k-1)}(|\ud i|+k-1)!\frac{y^iu(x)}{S^{|\ud i|+k}}\Omega_{\ud x}\wedge\Omega_{\ud y}\right]\\
&\quad\quad+\left[(-1)^{|\ud i|+n(k-1)}(|\ud i|+k-1)!\frac{y^iu(x)}{S^{|\ud i|+k}}(\theta_{ch}\Omega_{\ud y})\wedge dx_0\wedge\cdots\wedge dx_n\right]
\end{align*}
The second term vanishes, since
\begin{align*}
s_{\ud G}^\ast\left(\left[\frac{\ud y^{\ud i}u(\ud x)}{S^{|\ud i|+k}}(\theta_{ch}\Omega_{\ud y})\wedge dx_0\wedge\cdots\wedge dx_n\right]\right)=0
\end{align*}
due to the term $dx_0\wedge\cdots\wedge dx_n$. Here
$$
 s_{\ud G}^*: H^{n+k-1}_{dR}(\BP(\cE) \setminus X_S; \bC) \xrightarrow{\simeq}  H^{n+k-1}_{dR}(\BP^n \setminus X_{\ud G}; \bC)
 $$
 is the isomorphism induced from the section $s_{\ud G}$.
Thus
\begin{align*}
\bar \varphi_S([\ud y^{\ud i}u(\ud x)])=\varphi_S(\ud y^{\ud i}u(\ud x))=\left[(-1)^{n(k-1)+|\ud i|}(|\ud i|+k-1)!\frac{y^iu(x)}{S^{|\ud i|+k}}\Omega_{\ud x}\wedge\Omega_{\ud y}\right]=[\a(S,\ud y^{\ud i}u(\ud x) )].
\end{align*}
\end{proof}

We consider the natural epimorphism which appeared in \cite[Proposition 10]{Dim95},
\[
\delta_{\ud G}: H^n_{dR}(\BP^n \setminus D;\bC) \to H^{n+k-1}_{dR}(\BP^n\setminus X_{\ud G};\bC).
\]
The following proposition due to Dimca plays a key role in our deformation theory.
\begin{prop}[the proposition 10, \cite{Dim95}] \label{dc}
We have the following equality
\[
\delta_{\ud G} ([\b(S, \ud y^{\ud i} u (\ud x)])=s_{\ud G}^* ([\a(S, \ud y^{\ud i} u(\ud x)]).
\]
\end{prop}

\begin{lem} \label{alter}
The isomorphism $\bar J_{\ud G}:H_K^0(\cA) \to H^{n-k}_{dR,\pr}(X,\bC)$ is explicitly given as follows:
\[
\bar J_{\ud G}([\ud y^{\ud i}u(\ud x)] )= \Res_{\ud G}\left(\delta_{\ud G} ([\b(S, \ud y^{\ud i} u (\ud x)])\right),  \quad \ud y^{\ud i} u(\ud x) \in \cA_{c_{\ud G}}^0,
\]
where $[\cdot]$ means the cohomology class.
\end{lem}

\begin{proof}
From \eqref{gma}, the linear map $J_{\ud G}$ is given by the composite of the following sequence of isomorphisms:
 {\small{
 \begin{eqnarray*}
\qquad \quad \bar J_{\ud G}: {A}/\cK_{\ud G}  \xrightarrow{\bar \varphi_S}   H^{n+k-1}_{dR}(\BP(\cE) \setminus X_S; \bC) \xrightarrow{s_{\ud G}^*}  H^{n+k-1}_{dR}(\BP^n \setminus X_{\ud G}; \bC)
\xrightarrow{\Res_{\ud G}} H^{n-k}_{dR,\pr}(X_{\ud G}; \bC).
\end{eqnarray*}
 }}
This implies that
\begin{align}
\begin{split}
\bar J_{\ud G}([\ud y^{\ud i}u(\ud x)] )&
=\Res_{\ud G}\left( s_{\ud G}^* ([ \a(S, \ud y^{\ud i} u(\ud x))] )\right)\quad (\text{by \eqref{result}}) \\
&=\Res_{\ud G}\left( \delta_{\ud G} ([\b(S, \ud y^{\ud i} u (\ud x))]) \right)\quad (\text{by Proposition \ref{dc}}).
\end{split}
\end{align}
\end{proof}
Recall the set up: Let $X_{\ud U}\subset \BP^n$ be a smooth projective complete intersection which is deformed
from $X=X_{\ud G}$ by homogeneous polynomials $\ud H = (H_1(\ud x),$ $ \cdots, 
H_k(\ud x))$, i.e. $\ud U =(U_1(\ud x), $ $\cdots, U_k(\ud x))$ with $U_i(\ud x)=G_i(\ud x) + H_i(\ud x)$ for each $i =1, \cdots, k$ are the defining equations
for $X_{\ud U}$.

\begin{proof}\textbf{Proof of Theorem \ref{firthm}}:
The Laplace transform tells us that 
$$
\left(  \int_{0}^\infty \cdots \int_{0}^\infty  \ud y^{\ud i} u(\ud x) e^{S(\ud q)} dy_1 \cdots dy_k  \right) \Omega_{\ud x} 
= \frac{(-1)^p i_1 ! \cdots i_k !  u(\ud x)  \Omega_{\ud x}}{ G_1^{i_1+1} \cdots G_k^{i_k+1}}
=-\beta(S, \ud y^{\ud i} u(\ud x))),
$$
where $p =i_1 +\cdots + i_k$. 
Thus, for $\ud y^{\ud i}  u(\ud x) \in \cA_{c_{\ud G}}^0$, we have
\begin{eqnarray*}
\cC^{\ud G}_{\g} (\ud y^{\ud i} u(\ud x))
=\int_{\g} J_{\ud G}([\ud y^{\ud i}u(\ud x)] )
=\int_{\g} \Res_{\ud G} (\delta_{\ud G} ( \beta(S, \ud y^{\ud i} u(\ud x))))
\end{eqnarray*}
by Definition \ref{dp} of $\cC^{\ud G}_{\g}$.
Therefore\footnote{This means that the polynomial realization of the period integral can be understood as Feynman path integral for a 0-dimensional quantum field theory with action functional $S(\ud q)$.}
$$
\cC^{\ud G}_{\g} (f) =- \int_{\g} \Res_{\ud G} \circ \delta_{\ud G} \left(  \int_{0}^\infty \cdots \int_{0}^\infty  f e^{S(\ud q)} dy_1 \cdots dy_k \Omega_{\ud x}  \right), \quad f \in \cA_{c_{\ud G}}^0=A_{c_{\ud G}}.
$$
Recall that $S=\sum_{i=1}^k y_i G_i(\ud x)$ and $\G=\sum_{i=1}^k y_i H_i(\ud x)$. For $\ud y^{\ud i} u(\ud x) \in A_{c_{\ud G}}$, 
\begin{align*}
\cC_{\g}^{\ud U} (\ud y^{\ud i} u(\ud x)) 
& = \int_{\g} \Res_{\ud U} \left( \delta_{\ud U} (\b(S+\G, \ud y^{\ud i} u(\ud x)))  \right)\\
&= -\int_{\g}  \Res_{\ud U} \left( \delta_{\ud U} \left(\Omega_{\ud x}  \int_{0}^\infty \cdots \int_{0}^\infty  \ud y^{\ud i} u(\ud x) e^{S+\G} dy_1 \cdots dy_k   \right)    \right)   
\end{align*}
Hence
\begin{eqnarray*}
 \cC_{\g}^{\ud G} \left( \ud y^{\ud i} u(\ud x) e^\Gamma\right) 
 &=& \int_{\g} \Res_{\ud G} \left( \delta_{\ud G} (\b(S, e^\Gamma \cdot\ud y^{\ud i} u(\ud x)))  \right)\\
 &=& -\sum_{m=0}^{\infty} \int_{\g}  \Res_{\ud G} \left( \delta_{\ud G} \left(\Omega_{\ud x}  \int_{0}^\infty \cdots \int_{0}^\infty  \frac{\Gamma^m}{m!} \cdot y^{\ud i} u(\ud x)  e^{S} dy_1 \cdots dy_k   \right)    \right)  
\end{eqnarray*}
converges to $\cC_{\g}^{\ud U} (\ud y^{\ud i} u(\ud x))$.
%
%
\end{proof}

See Proposition \ref{klemma} for a modern deformation theoretic interpretation of Theorem \ref{firthm}.

Finally, we justify Remark \ref{fourrem} $(1)$ saying that $J_{\ud G}=\Res_{\ud G} \circ s_{\ud G}^* \circ \varphi_S$ is defined in the level of cochains. Since we already saw $\varphi_S$ and $s_{\ud G}^*$ are defined in the level of cochains, it remains to deal with the residue map $\Res_{\ud G}$. 

Let $T$ be a sufficiently small closed tubular neighborhood  of $X_{\ud G}(\bC)$ in $\BP^n(\bC)$.
Let $\partial T$ be the boundary of $T$. Let $\rho: T \to X_{\ud G}(\bC)$ be the (topological) projection map.\footnote{In general, the projection $\rho$ is not holomorphic; see $(i)$ on the page 96 and Remark 15 of \cite{Dim95}.} Then $\rho$ restricted to $\partial T$ is an $S^{2k-1}$-sphere bundle over $X_{\ud G}(\bC)$.
If $j_{\partial T}: \partial T \hookrightarrow \BP^n(\bC) \setminus X_{\ud G}(\bC)$ is the inclusion map, then the residue map 
$$
\Res_{\ud G}: \Omega^s (\BP^n\setminus X_{\ud G}) \to \Omega^{s-2k+1}(X_{\ud G}), \quad s=2k-1, \cdots, 2n,
$$
is defined by
\eqn\resdef{
\Res_{\ud G} := \rho_* \circ j_{\partial T}^*,
}
where $j_{\partial T}^*$ is the natural pull-back of differential forms and $\rho_*$ is integration along the fibers
of $\rho:\partial T \to X_{\ud G}(\bC)$. 
This residue map is well-defined on the level of differential forms; A. Dimca considered this phenomenon surprising (see page 96 of \cite{Dim95} for details).
This induces an isomorphism
$$
\Res_{\ud G}: H_{dR}^{n+k-1}(\BP^n\setminus X_{\ud G};\bC) \xrightarrow{\simeq} H_{dR, \pr}^{n-k}(X_{\ud G};\bC).
$$

\subsection{The descendant functor and a formal deformation theory} \label{sec3.3}

This subsection we will study formal deformations of the data $(\cB,\cdot,K)$ and $C$ where
\begin{quote}
(1) $(\cB,\cdot)$ is a $\mathbb{Z}$-graded super-commutative associative algebra algebra over a field $k$,\\
(2) $(\cB,K)$ is a cochain complex over $k$, and\\
(3) $C:(\cB,K)\rightarrow (k,0)$ is a $k$-linear cochain map, where $(k,0)$ is the cochain complex with zero differential.
\end{quote}
%
%
%
Let $\mathfrak{C}_k$ be the category of triples $(\cB,\cdot,K)$ satisfying (1) and (2) above with ($k$-linear) cochain maps.
The deformation theory of objects in $\mathfrak{C}_k$ was studied in \cite{PP16}. We briefly review it here.

\begin{defn}\label{partition}
A \emph{partition} $\pi= B_1 \cup B_2\cup \cdots$ of the set $[n]=\{1,2, \cdots, n\}$ is a decomposition of $[n]$ into a pairwise disjoint non-empty subsets $B_i$, called \emph{blocks}. Blocks are ordered by the minimum element of each block and each block is ordered by the ordering induced from the ordering of natural numbers. The notation $|\pi|$ means the number of blocks in a partition $\pi$ and $|B|$ means the size of the block $B$. If $k$ and $k'$ belong to the same block in $\pi$, then we use
the notation $k \sim_\pi k'$. Otherwise, we use $k \nsim_\pi k'$. Let $P(n)$ be the set of all partitions of $[n]$. We use the following notation:
\be
x_B&=& x_{j_1} \otimes \cdots \otimes x_{j_{r}} \text{ if }  B=\{j_1, \cdots, j_{r}\},\\
\phi^f(x_{B})&=&\phi^f_r(x_{j_1}, \cdots, x_{j_r}) \text{ if } B=\{j_1, \cdots, j_r\}.
\ee
\end{defn}

\bed \label{defdesc}
For a given object $(\cB, \cdot, K)$ in $\frC_k$, we define $\Des\left(\cB, \cdot, K\right) = (\cB, \underline \ell^K)$, where $\underline \ell^K=\ell_1^K, \ell_2^K, \cdots$ is the family of linear maps $\ell_n^K: S^n(\cB) \to \cB$, inductively defined by the formula: $\ell_1^{K}=K$ and
\be
&&\ell_n^{K}(x_1, \cdots, x_{n-1}, x_n)=  \ell_{n-1}^{K}(x_1,\cdots, x_{n-2}, x_{n-1}\cdot x_n)\\
&&
-\ell_{n-1}^{K}(x_1, \cdots, x_{n-1}) \cdot x_{n} 
-(-1)^{|x_{n-1}|(1+|x_1|+\cdots + |x_{n-2}|)}  x_{n-1}\cdot \ell_{n-1}^{K}(x_1, \cdots, x_{n-2}, x_n), \quad n \geq 2,
\ee
for any homogeneous elements $x_1, x_2, \cdots, x_n \in \cB$.

 For a given morphism $f: (\cB, \cdot, K) \to (\cB', \cdot, K')$ in $\frC_k$, we define $\Des(f)= \underline \phi^f=\phi_1^f, \phi_2^f, \cdots $ as a family of $k$-linear maps $\phi_n^f: S^n(\cB) \to \cB'$ defined inductively by the formula: $\phi_1^{f}= {f}$ and
 \be
\phi_m^{f}(x_1, \cdots, x_m) = \phi_{m-1}^{f} (x_1, \cdots, x_{m-2}, x_{m-1}\cdot x_m) -
\sum_{\substack{\pi \in P(m), |\pi|=2 \\ m-1 \nsim_\pi m}} \phi^{f} (x_{B_1})\cdot \phi^{f}(x_{B_2}), \quad m \geq 2,
\ee
 for any homogeneous elements $x_1, x_2, \cdots, x_m \in \cB$. 
  Here we use the following notation:
\be
x_B&=& x_{j_1} \otimes \cdots \otimes x_{j_{r}} \text{ if }  B=\{j_1, \cdots, j_{r}\},\\
\phi^f(x_{B})&=&\phi^f_r(x_{j_1}, \cdots, x_{j_r}) \text{ if } B=\{j_1, \cdots, j_r\}.
\ee
\eed

\begin{prop}
Let $\mL_k$ be the category of $L_\infty$-algebras. The above assignment $\Des$ is a functor from $\frC_k$ to $\mL_k$.
\end{prop}
\begin{proof}
See subsection 3.2 in \cite{PP16}.
\end{proof}

We will call $\ud \ell^K$ and $\ud \phi^f$ a \textit{descendant} $L_\infty$-algebra and $L_\infty$-morphism respectively.
In short, the functor $\Des:\frC_k \to \mL_{k}$ takes

(i)
an object $(\cB,\hbox{}\cdot\hbox{}, K)$ in $\frC_k$ to a descendant $L_\infty$-algebra
$(\cB,\underline{\ell}^{K}=\ell_1^K,\ell_2^K, \ell_3^K,\cdots)$, where $\ell_1^K=K$ and $\ell^K_2, \ell_3^K,\cdots$
measure the failure and higher failures of $K$ being a derivation of the multiplication in $\cB$,

(ii) a morphism $f$ in $\mC_k$ to a descendant $L_\infty$-morphism $\underline{\phi}^f=\phi_1^f, \phi_2^f,\phi_3^f,\cdots$, where $\phi_1^f=f$ and $\phi_2^f,\phi_3^f,\cdots$ measure
the failure and higher failures of $f$ being an algebra homomorphism.

\begin{rem}
The functor $\Des$ can be viewed as a generalization of the BV construction \eqref{elltwo} for a differential operator $L$ of order $\leq 2$ to a differential operator $K$ of any finite order. Note that the $L_\infty$-algebra $(\cB, K, 0, 0, \cdots )$ is isomorphic to $\Des(\cB, \cdot, K)= (\cB, \ell_1^K=K, \ell_2^K, \ell_3^K, \cdots ) $ as $L_\infty$-algebras.
An important point for our deformation theory is to choose a \textit{nice} representative, given by the descendant $L_\infty$-algebra, in the $L_\infty$-isomorphism class of $(\cB, K, 0, 0, \cdots )$.
\end{rem}

\vspace{1em}
For $\mathfrak{a}\in\art_k$ denote $\mathfrak{m_a}$ its maximal ideal. In what follows we endow $\mathfrak{a}\otimes\cB$ the natural $\mathbb{Z}$-grading.
We need the following lemmas to see a precise relationship between a formal deformation theory for $\frC_k$ (based on the Maurer-Cartan equation) and the descendant functor $\Des$.
\begin{lem} \label{proofofse}
For $\Gamma\in(\mathfrak{m_a}\otimes\cB)^0$ and homogeneous $\lambda\in\mathfrak{a}\otimes\cB$, denote
\begin{align*}
L^K(\Gamma):=\sum_{n\geq1}\frac{1}{n!}\ell^K_n(\Gamma,\cdots,\Gamma),\quad L^K_\Gamma(\lambda):=K\lambda+\sum_{n\geq2}\frac{1}{(n-1)!}\ell^K_n(\Gamma,\cdots,\Gamma,\lambda).
\end{align*}
Then we have the following identities:
\begin{align*}
K(e^\Gamma-1)=L^K(\Gamma)\cdot e^\Gamma,\quad K(\lambda\cdot e^\Gamma)=L^K_\Gamma(\lambda)\cdot e^\Gamma+(-1)^{|\lambda|}\lambda\cdot K(e^\Gamma-1).
\end{align*}
\end{lem}
\begin{proof}
See \cite[Lemma 3.1]{PP16}.
\end{proof}

\begin{lem} \label{proofofsec}
For $\Gamma\in(\mathfrak{m_a}\otimes\cB)^0$ and homogeneous $\lambda\in\mathfrak{a}\otimes\cB$, denote
\begin{align*}
\Phi^f(\Gamma):=\sum_{n\geq1}\frac{1}{n!}\phi^f_n(\Gamma,\cdots,\Gamma),\quad \Phi^f_\Gamma(\lambda):=\phi^f_1(\lambda)+\sum_{n\geq2}\frac{1}{(n-1)!}\phi^f_n(\Gamma,\cdots,\Gamma,\lambda).
\end{align*}
Then we have the following identities:
\begin{align*}
f(e^\Gamma-1)=e^{\Phi^f(\Gamma)}-1,\quad f(\lambda\cdot e^\Gamma)=\Phi^f_\Gamma(\lambda)\cdot e^{\Phi^f(\Gamma)}.
\end{align*}
\end{lem}
\begin{proof}
See \cite[Lemma 3.3]{PP16}.
\end{proof}

We constructed the GBV algebra $(\cA, \cdot, K_{\ud G}, \ell_2^{K_{\ud G}})$ from $X_{\ud G}$
and the cochain enhancement (the polynomial realization) $\cC_\g^{\ud G}: (\cA, K_{\ud G}) \to (\C, 0)$ of the period integral $C_{\g}^{\ud G}$ for a fixed homology cycle $[\g]\in H_{n-k}(X_{\ud G}(\bC),\bZ)_0$. The key idea for studying deformations of period integrals of $X_{\ud G}$ is to study the algebraic deformation of $(\cA, \cdot, K_{\ud G}, \ell_2^{K_{\ud G}})$ and $\cC_\g^{\ud G}$.

Note that $\G=\sum_{i=1}^k y_i H_i(\ud x) \in A$ (where $\ud H= (H_1(\ud x), \cdots, H_k(\ud x))$ is the deformation data introduced earlier) is a solution to the shifted Maurer-Cartan equation for the dgla $(\cA, K_{\ud G}, \ell_2^{K_{\ud G}})$:
\begin{eqnarray}\label{mce}
K_{\ud G}(e^\G-1):=K_{\ud G}(\G) + \frac{1}{2} \ell_2^{K_{\ud G}}(\G, \G)=0.
\end{eqnarray}
Define
\begin{eqnarray} \label{kgamma}
K_\G(\lambda):=L_{\G}^{K_{\ud G}}(\lambda)= {K_{\ud G}}(\lambda) +\ell_2^{K_{\ud G}}\big(\G, \lambda \big),
\quad \lambda \in \cA_{}
\end{eqnarray}
where we use the computational fact $\ell_r^{K_{\ud G}}=0$ for $r\geq 3$ in the second equailty.
By Proposition \ref{proofofse},
the operator $K_\G$ is a $\C$-linear map on $\cA_{}$ of degree 1 and satisfies
\be
K_\G^2=0.
\ee
Also define
$$
Q_\G(\cdot):=\ell_2^{K_{\ud G}} (S+\G,\cdot  ), \quad \text{ where } S+\G= \sum_{i=1}^k y_i (G_i(\ud x) + H_i(\ud x)).
$$

A direct computation implies the following proposition (recall that $\ud U= \ud G + \ud H$).
\begin{prop}
The quadruple $(\cA, \cdot, Q_\G, K_\G, \ell_2^{K_\G})$ is the DGBV algebra $(\cA, \cdot, Q_{\ud U}, K_{\ud  U}, \ell_2^{K_{\ud U}})$ associated to the smooth projective complete intersection $X_{\ud U}$.
\end{prop}

In short, if we like to deform the cochain complex $(\cA, K)$, then we consider it as an object of $\mL_k$ and find an $L_\infty$-algebra
$(\cA, K, \ell_2, \ell_3, \cdots)$ which is $L_\infty$-isomorphic to $(\cA, K, \ud 0)$ and  whose higher Lie bracket $\ell_m$ ($m\geq 2$)  is non-trivial. The functor $\Des$ is designed to give such an explicit $L_\infty$-algebra $(\cA, K=\ell_1^K, \ell_2^K, \cdots)$ using the auxiliary data, some nontrivial multiplication on $\cA$ (see Definition \ref{defdesc} for how we use this nontrivial multiplication). 
Then we obtain a nontrivial deformation $(\cA,K_\Gamma)$ of the cochain complex $(\cA,K)$ as above using the Maurer-Cartan solution $\G=\sum_{i=1}^k y_i H_i(\ud x)$.
 We can also deform a cochain map $\cC_\g^{\ud G}:(\cA, K_{\ud G}) \to (\bC,0)$ by using the Maurer-Cartan solution $\G$.

\bep \label{klemma}
Let $\cC: (\cA, K_{\ud G}) \to (\bC,0)$ be any cochain map.
Let $\G =\sum_{i=1}^k y_i H_i(\ud x)\in A$ be the solution of the Maurer-Cartan equation \eqref{mce}. If we define\footnote{Here $e^\G=1 + \G + \frac{\G \cdot \G}{2} + \cdots $ and we think of
$\cC(x \cdot e^\G)=\cC(x) + \cC(x\cdot \G) + \cC(x\cdot \frac{\G \cdot \G}{2})+\cdots $ as a formal expression.}
\be
\cC_{\G} (x) := \cC\left( x \cdot e^\G\right), \quad x\in  \cA,
\ee
then $\cC_{\G}: (\cA, K_\G)=(\cA, K_{\ud U}) \to (\C,0)$ is a cochain map.
\eep

\begin{proof}
We have to check that $\cC_\G \circ K_\G=0$. This follows from
\begin{eqnarray*}
(\cC_{\G} \circ K_\G )(x) 
&=& \cC \left(K_\G(x) \cdot e^\G\right) \\
&=& \cC  \left(K_{\ud G} \left(x \cdot e^\G\right)\right)  \quad \text{(by Lemma \ref{proofofse})}\\
&=&0  
\end{eqnarray*}
for any element $x  \in \cA$.
\end{proof}

\begin{rem}
This proposition \ref{klemma} motivates the statement of Theorem \ref{firthm} and can be viewed as its modern deformation theoretic interpretation via the shifted differential graded Lie algebra $(\cA, K_{\ud G}, \ell_2^{K_{\ud G}})$ and the DGBV algebra $\bvx=(\cA, \cdot, Q_{\ud G}, K_{\ud  G}, \ell_2^{K_{\ud G}})$. Note that Proposition \ref{klemma} is not refined enough to prove Theorem \ref{firthm}.
\end{rem}

\subsection{The Bell polynomials and proof of Theorem \ref{sthm}}\label{sec3.4}

The \emph{(complete) Bell polynomials} $B_n(x_1,\cdots,x_n)$ are defined by the power series expansion
\begin{align}\label{Bpoly}
\exp\left(\sum_{i\geq1}x_i\frac{t^i}{i!}\right)=1+\sum_{n\geq1}B_n(x_1,\cdots,x_n)\frac{t^n}{n!}.
\end{align}
The \emph{partial Bell polynomials} $B_{n,k}(x_1,\cdots,x_{n-k+1})$ are defined by the power series expansion
\begin{align*}
\exp\left(u\sum_{i=1}^\infty x_i\frac{t^i}{i!}\right)=\sum_{n,k\geq0}B_{n,k}(x_1,\cdots,x_{n-k+1})\frac{t^n}{n!}u^k=1+\sum_{n\geq1}\frac{t^n}{n!}\left(\sum_{k=1}^nu^kB_{n,k}(x_1,\cdots,x_{n-k+1})\right)
\end{align*}
which gives the following formula:
\begin{align*}
B_{n,k}(x_1,\cdots,x_{n-k+1})=\sum\frac{n!}{i_1!\cdots i_{n-k+1}!}\left(\frac{x_1}{1!}\right)^{i_1}\cdots\left(\frac{x_{n-k+1}}{(n-k+1)!}\right)^{i_{n-k+1}}
\end{align*}
where the sum runs over all sequences $i_1,\cdots,i_{n-k+1}$ of non-negative integers such that
\begin{align*}
i_1+\cdots+i_{n-k+1}=k,\quad i_1+2i_2+\cdots+(n-k+1)i_{n-k+1}=n.
\end{align*}
Then the following formula holds:
\begin{align*}
B_n(x_1,\cdots,x_n)=\sum_{k=1}^nB_{n,k}(x_1,\cdots,x_{n-k+1}),\quad B_0=1.
\end{align*}
There are recurrence relations
\begin{align*}
B_{n+1}(x_1,\cdots,x_{n+1})=\sum_{i=0}^n\binom{n}{i}B_{n-i}(x_1,\cdots,x_{n-i})x_{i+1}
\end{align*}
and
\begin{align*}
B_{n,k}(x_1,\cdots,x_{n-k+1})=\sum_{i=1}^{n-k+1}\binom{n-1}{i-1}x_iB_{n-i,k-1}(x_1,\cdots,x_{n-i+k})
\end{align*}
where $B_{0,0}=1$, $B_{n,0}=0$ for $n\geq1$, and $B_{0,k}=0$ for $k\geq1$. For example, $B_1(x_1)=x_1, B_2(x_1, x_2)=x_1^2+x_2$, $B_3(x_1,x_2,x_3)=x_1^3+3x_1x_2+x_3, \cdots$.

\begin{proof} \textbf{Proof of Theorem \ref{sthm}}:
We have
\begin{align*}
\cC_{\g}^{\ud U}(u)
&= \cC_{\g}^{\ud G} \left(u \cdot e^\Gamma\right) \quad (\text{by Theorem \ref{firthm}})
\\
&= \Phi_\G^{\cC_\a} (u) \cdot  e^{\Phi^{\cC_\a}(\G)}\quad (\text{by Lemma \ref{proofofsec} and $\cC_\a= \cC_{\g}^{\ud G} $}).
\end{align*}

In order to prove Theorem \ref{sthm}, we need to expand
$\cC_{\g}^{\ud U}(u)$ as
\begin{align*}
\cC_{\g}^{\ud U}(u)=\Psi_0(\Gamma)+\Psi_1(\Gamma)+\Psi_2(\Gamma)+\Psi_3(\Gamma)+\cdots,
\end{align*}
where $\Psi_n(c  \Gamma)=c^n\Psi_n(\Gamma)$ for $c\in\C$. 
By the definition of the Bell polynomials, we have
\begin{align*}
e^{\Phi^{\cC_\a}(\G)}
=\exp\left(\sum_{n\geq1}\frac{1}{n!}\phi_n^{\cC_\a}(\Gamma, \cdots, \Gamma)\right)=1+\sum_{n\geq1}\frac{1}{n!}B_n\left(\phi^{\cC_\a}_1(\Gamma),\cdots,\phi^{\cC_\a}_n(\Gamma,\cdots, \Gamma)\right).
\end{align*}
This expansion has an advantage that
\begin{align*}
B_n\left(\phi^{\cC_\a}_1(c\cdot\Gamma),\cdots,\phi^{\cC_\a}_n(c\cdot\Gamma,\cdots, c \cdot \Gamma)\right)=
c^n\cdot B_n\left(\phi^{\cC_\a}_1(\Gamma),\cdots,\phi^{\cC_\a}_n(\Gamma, \cdots, \Gamma)\right).
\end{align*}
for $c\in\C$. We also have
$$
\Phi^{\cC_\a}_\Gamma(u)=\phi^{\cC_\a}_1(u)+\sum_{n\geq2}\frac{1}{(n-1)!}\phi^{\cC_\a}_n(\Gamma,\cdots,\Gamma,u).
$$
Thus
\begin{eqnarray*}
\cC_{\g}^{\ud U}(u)&= &  \Phi_\G^{\cC_\a} (u) \cdot  e^{\Phi^{\cC_\a}(\G)}  \\
& = &
\cC_\a(u)+\sum_{m\geq1}\sum_{\substack{j+k=m \\ j,k\geq0}}\frac{1}{j!k!} B_j\left(\phi^{\cC_\a}_1(\Gamma),\cdots,\phi^{\cC_\a}_j(\Gamma, \cdots, \Gamma)\right)\cdot\phi^{\cC_\a}_{k+1}(\Gamma, \cdots, \Gamma,u)
\end{eqnarray*}
which finishes the proof.
\end{proof}


\subsection{Deformation formula for the period matrices and proof of Theorem \ref{tthm}}\label{sec3.5}
Recall that $I'$ is the set of indices $i$ such that $H_i(\ud x) \neq 0$.
We assume that the cardinality of $I$ is bigger than $|I'|$ and 
view $I'$ as a subset of $I$ (allowing a slight abuse of notation). Let $\ell = |I'| \leq m$.
For each $i=1, \cdots, \ell$, we choose $h(\ud x) \in \bQ[\ud x]$ such that $\ch(y_i H_i(\ud x)h(x)) =c_{\ud G}$, when $c_{\ud G} >0$. When $c_{\ud G} <0$, we choose $j \in \{ 1, \cdots, k\}$ and a positive integer $m$ such that 
$$
m d_j = \min \{ m_i d_i : \deg( h(\ud x))=m_i d_i  + c_{\ud G} \geq 0, m_i \text{ is a positive integer}, i=1, \cdots, k \}
$$
 so that $\ch(y_i H_i(\ud x) y_j^{m} h(\ud x))=c_{\ud G} < 0.$
 Then we define
\begin{eqnarray}\label{ualpha}
u_{\a_i}:=
\begin{cases}
y_i H_i(\ud x), \quad \text { if } c_{\ud G}=0,
 \\
y_i H_i(\ud x)h(\ud x), \quad \text { if } c_{\ud G}>0
 \\
y_i H_i(\ud x) y_j^{m} h(\ud x), \quad \text { if } c_{\ud G}<0
 \end{cases}
\end{eqnarray}
 for $i =1, \cdots, \ell$.
Now we assume that
\begin{align}\label{aind}
\{u_{\a_i}  + \cK_{\ud U} : i =1, \cdots, \ell \} \text{ is $\Q$-linearly independent in } \bQ[\ud q]/\cK_{\ud U}.
\end{align}

\begin{lem}\label{cb}
There exists a subset $\{u_\a : \a \in I\} \subset A_{c_{\ud G}}$ which includes $\{u_{\a_i}: i =1, \cdots, \ell \} $  such that 
$$
\{ u_\a \mod K_{\ud U}(\cA^{-1}) \} \text{ is a $\bQ$-basis of } H_{K_{\ud U}}^0(\cA),
$$
where we use the notation $I'=\{\a_{i_1}, \cdots, \a_{i_\ell}\}  \subseteq I$ in order to view $I'$ as a subset of $I$.
\end{lem}

\begin{proof}
Because of the condition \eqref{aind}, we can extend $\{u_{\a_{i}} \mod K_{\ud U}(\cA^{-1}) : \a_i \in I' \}$ to a $\bQ$-basis
$\{ u_\a \mod K_{\ud U}(\cA^{-1}) : \a \in I \}$ of $H_{K_{\ud U}}^0(\cA)$.
\end{proof}

In the definition \eqref{pma} of the period matrix , we chose $e_\a \in A_{c_{\ud G}}=\cA^0_{c_{\ud G}}$ such that 
$$
\{ e_\a \mod K_{\ud G}(\cA^{-1}) \} \text{ is a $\bC$-basis of } H_{K_{\ud G}}^0(\cA).\
$$
For each $\rho \in I$, we define a power series $T^\r (\ud t) \in \bC[[\ud t]]$ by the following formula 
\begin{eqnarray}\label{eaon}
\qquad \qquad
 \sum_{\r \in I} T^\r(\ud t) \cdot e_\r + K_{\ud G} (\Lambda(\ud t))=
\begin{cases}
e^{\sum_{\a\in I} t^{\a} u_\a}-1, \quad \text { if } c_{\ud G}=0,
 \\
\left(h(\ud x)+ \sum_{\b \in  I\setminus I'} t^\b u_\b  \right)e^{\sum_{i=1}^k t^{\a_i} y_i H_i(\ud x)}, \quad  \text { if } c_{\ud G} >0,
 \\
\left(y_j^{m}h(\ud x)+ \sum_{\b \in  I\setminus I'} t^\b u_\b  \right)e^{\sum_{i=1}^k t^{\a_i} y_i H_i(\ud x)}, \quad \text { if } c_{\ud G}<0,
 \end{cases}
\end{eqnarray}
for some $\Lambda(\ud t) \in \cA^{-1}[[\ud t]]$. 
Note that \eqref{eaon} uniquely determines $T^\r (\ud t)$. Note that 
$$
\sum_{i=1}^k t^{\a_i} y_i H_i(\ud x)=\sum_{i=1}^\ell t^{\a_i} y_i H_i(\ud x).
$$

\begin{proof} \textbf{Proof of Theorem \ref{tthm}:}

We first deal with the case $c_{\ud G}=0$. For each $\b \in I$, we take a partial derivative $\partial_\b$ of \eqref{eaon}:
$$
\partial_\b \left(e^{\sum_{\a} t^{\a} u_\a}-1
\right)=
\sum_{\r \in I} \left(  \partial_\b T^\r(\ud t) \right) \cdot e_\r + K_{\ud G} (\partial_\b\Lambda(\ud t)).
$$
Then we have that
$$
u_\b e^{\G}= \sum_{\r \in I} \left(  \partial_\b T^\r(\ud t) \big|_{\substack{t^\e=1, \e \in I'\\t^\e=0, \e \in I\setminus I'}}\right) \cdot e_\r + K_{\ud G} (\partial_\b\Lambda(\ud t)\big|_{\substack{t^\a=1, \a \in I'\\t^\a=0, \a \in  I\setminus I'}}),
$$
where $\G=\sum_{i=1}^k y_i H_i(\ud x)$.
Now we take $\cC_{\g_\a^{\ud G}}$ to obtain
$$
\cC_{\g_\a^{\ud G}} (u_\b e^{\G})
=
\sum_{\r \in I} \left(  \partial_\b T^\r(\ud t) \big|_{\substack{t^\e=1, \e \in I'\\t^\e=0, \e \in  I\setminus I'}}\right) \cdot \cC_{\g_\a^{\ud G}}  (e_\r ),
$$
since $\cC_{\g_\a^{\ud G}}  \circ K_{\ud G} =0$.
Therefore,
\begin{eqnarray*}
\cC_{\g^{\ud U}_\d} (u_\b) &=& \sum_{\a \in I}{B^\a_\d}\cdot \cC_{\g^{\ud G}_\a} \left(u_\b e^\Gamma\right) \quad (\text{by Theorem \ref{firthm} and Remark \ref{ccc} $(c)$})
\\
&=&\sum_{\a \in I} \sum_{\r \in I} \left(  \partial_\b T^\r(\ud t) \big|_{\substack{t^\e=1, \e \in I'\\t^\e=0, \e \in  I\setminus I'}}\right) \cdot \cC_{\g_\a^{\ud G}}  (e_\r ) \cdot {B^\a_\d},
\end{eqnarray*}
which implies the theorem.

Now assume that $c_{\ud G} > 0$. For each $\b \in I$, we take a partial derivative $\partial_\b$ of \eqref{eaon}:
$$
\partial_\b \left( \big(h(\ud x)+ \sum_{\d \in  I\setminus I'} t^\d u_\d \big) 
e^{\sum_{i=1}^k t^{\a_i} y_i H_i(\ud x)}
\right)=
\sum_{\r \in I} \left(  \partial_\b T^\r(\ud t) \right) \cdot e_\r + K_{\ud G} (\partial_\b\Lambda(\ud t)).
$$
When $\b\in I \setminus I'$, the LHS becomes
$$
\left( u_{\b} e^{\sum_{i=1}^k t^{\a_i} y_i H_i(\ud x)}\right) \Big|_{\substack{t^\e=1, \e \in I'\\t^\e=0, \e \in  I\setminus I'}}= u_\b e^{\G}, \quad \G=\sum_{i=1}^k y_i H_i(\ud x).
$$
When $\b=\a_{i} \in I'$, the LHS becomes
$$
\left( \big(h(\ud x)+ \sum_{\d \in  I\setminus I'} t^\d u_\d \big) y_i H_i(\ud x) e^{\sum_{i=1}^k t^{\a_i} y_i H_i(\ud x)}\right) \Bigg|_{\substack{t^\e=1, \e \in I'\\t^\e=0, \e \in  I\setminus I'}}= u_\b e^{\G} \quad \text{by } \eqref{ualpha}.
$$
Therefore we get that
$$
u_\b e^{\G}= \sum_{\r \in I} \left(  \partial_\b T^\r(\ud t) \big|_{\substack{t^\e=1, \e \in I'\\t^\e=0, \e \in  I\setminus I'}}\right) \cdot e_\r + K_{\ud G} (\partial_\b\Lambda(\ud t)\big|_{\substack{t^\e=1, \e \in I'\\t^\e=0, \e \in  I\setminus I'}},
$$
for each $\b\in I$. Then the same argument as above gives the desired result
in the case of $c_{\ud G} >0$.
The case of $c_{\ud G} <0$ can be treated in a similar way.
\end{proof}

\appendix
\section{}\label{apa}
\begin{prop}\label{CI-ambient-isotopy} If $X_{\underline{G}}$ and $X_{\underline{U}}$ are smooth projective complete intersections in $\mathbf{P}^n$ defined by homogeneous polynomials $\underline{G}=(G_1(\underline{x}),\cdots,G_k(\underline{x}))$ and $\underline{U}=(U_1(\underline{x}),\cdots,U_k(\underline{x}))$ of degree $(d_1,\cdots,d_k)$ respectively, then there is an ambient isotopy\footnote{Given two smooth embeddings $\alpha,\beta:L\hookrightarrow M$ of smooth manifolds, an \emph{ambient isotopy} from $\alpha$ to $\beta$ is a smooth map $F:[0,1]\times M\rightarrow M$ such that (1) $F(0,\cdot)=\mathrm{Id}_M$, (2) $F(t,\cdot):M\rightarrow M$ is a diffeomorphism for every $t\in[0,1]$, and (3) $F(1,\cdot)\circ\alpha=\beta$.} $F:[0,1]\times\mathbf{P}^n(\mathbb{C})\rightarrow\mathbf{P}^n(\mathbb{C})$ such that the induced diffeomorphism $F(1, \cdot):\mathbf{P}^n(\mathbb{C})\rightarrow\mathbf{P}^n(\mathbb{C})$ restricts to give a diffeomorphism from $X_{\underline{G}}(\mathbb{C})$ to $X_{\underline{U}}(\mathbb{C})$. Consequently, there is an isomorphism on homologies
\begin{align*}
\xymatrix{F(1,\cdot)_\ast:H_\bullet(X_{\underline{G}}(\bC),\mathbb{Z}) \ar[r]^-\sim & H_\bullet(X_{\underline{U}}(\bC),\mathbb{Z})}
\end{align*}
which fixes cycles supported in $X_{\underline{G}}(\mathbb{C})\cap X_{\underline{U}}(\mathbb{C})$.
\end{prop}
\begin{proof}
We follow the argument in \cite[section 4]{KulWo} which covers the hypersurface case. The polynomials
\begin{align*}
t_0\underline{G}+t_1\underline{U}=(t_0G_1+t_1U_1,\cdots,t_0G_k+t_1U_k)
\end{align*}
of bidegree $(1,d_1),\cdots,(1,d_k)$ define a complete intersection $W\subseteq\mathbf{P}^n(\mathbb{C})$. The set $C\subseteq\mathbf{P}^1(\mathbb{C})\times\mathbf{P}^n(\mathbb{C})$ of points such that $W\cap([z_0:z_1]\times\mathbf{P}^n(\mathbb{C}))$ is singular is a closed algebraic set. Let $\pi:\mathbf{P}^1(\mathbb{C})\times\mathbf{P}^n(\mathbb{C})\rightarrow\mathbf{P}^1(\mathbb{C})$ be the projection. Then $\pi(C)\subseteq\mathbf{P}^1(\mathbb{C})$ is also a closed algebraic set. Since $X_{\underline{G}}(\mathbb{C})$ and $X_{\underline{U}}(\mathbb{C})$ are nonsingular, $\pi(C)\neq\mathbf{P}^1(\mathbb{C})$ so $\pi(C)$ is a finite set of points. Let $I$ be a smooth arc in $\mathbf{P}^1(\mathbb{C})$ from $[1:0]$ to $[0:1]$ in the complement of $\pi(C)$. Then $\pi^{-1}(I)=I\times\mathbf{P}^n(\mathbb{C})$ contains the smooth submanifold $\pi^{-1}(I)\cap W$ of real codimension $2k$. By construction the projection $\pi$ is a proper submersion, where the properness follows from the compactness of $\pi^{-1}(I)\cap W\subseteq I\times\mathbf{P}^n(\mathbb{C})$. Since $I$ is contractible, there is a diffeomorphism $f$ over $I$:
\begin{align*}
\xymatrixcolsep{1pc}\xymatrix{
I\times X_{\underline{G}}(\mathbb{C}) \ar[dr]_-{\mathrm{pr}_I} \ar[rr]^-\sim_-f & & \pi^{-1}(I)\cap W \ar[dl]^-\pi \\
& I &
}
\end{align*}
given by Ehresmann's theorem (see \cite[Theorem 9.3]{Vois} for example). If we denote the vector field
\begin{align*}
\frac{\partial}{\partial t}\in\Gamma\left(I\times X_{\underline{G}}(\mathbb{C}),T(I\times X_{\underline{G}}(\mathbb{C}))\right)
\end{align*}
tangential to the first factor $I$, then the pushforward
\begin{align*}
\varphi_\ast\frac{\partial}{\partial t}\in\Gamma\left(\pi^{-1}(I)\cap W,T\left(\pi^{-1}(I)\cap W\right)\right)
\end{align*}
extends to a vector field $\tau$ on $I\times\mathbf{P}^n(\mathbb{C})$ (for example, extension by using the partition of unity argument on $I\times\mathbf{P}^n(\mathbb{C})$ see \cite[II. 3]{OW}). Then the integral flow of $\tau$ gives an ambient isotopy in $\mathbf{P}^n(\mathbb{C})$ taking $X_{\underline{G}}(\mathbb{C})$ to $X_{\underline{U}}(\mathbb{C})$ which restricts to a diffeomorphism from $X_{\underline{G}}(\mathbb{C})$ to $X_{\underline{U}}(\mathbb{C})$

Now, denote $F(t,\cdot):[0,1]\times\mathbf{P}^n(\mathbb{C})\rightarrow\mathbf{P}^n(\mathbb{C})$ an isotopy constructed as above and $M$ the common underlying smooth manifold of both $X_{\underline{G}}(\mathbb{C})$ and $X_{\underline{U}}(\mathbb{C})$. then the two subvarieties are realized as two embeddings
\begin{align*}
\xymatrix{
M \ar[d]_-{\mathrm{Id}_M} \ar[r]^-\sim_-{\iota_{\underline{G}}} & X_{\underline{G}}(\mathbb{C}) \ar[d]^-{F(1,\cdot)|_{X_{\underline{G}}}(\mathbb{C})} \ar@{^(->}[r]^-{\textrm{incl.}} & \mathbf{P}^n(\mathbb{C}) \ar[d]^-{F(1,\cdot)} \\
M \ar[r]^-\sim_-{\iota_{\underline{U}}} & X_{\underline{U}}(\mathbb{C}) \ar@{^(->}[r]^-{\textrm{incl.}} & \mathbf{P}^n(\mathbb{C})
}
\end{align*}
such that $F(1,\cdot)\circ\iota_{\underline{G}}=\iota_{\underline{U}}$. When we say that a singular simplex $\gamma$ belongs to $X_{\underline{G}}(\mathbb{C})\cap X_{\underline{U}}(\mathbb{C})$, this means that $\gamma:\Delta^r\rightarrow\mathbf{P}^n(\mathbb{C})$ factors through both $\iota_{\underline{G}}$ and $\iota_{\underline{U}}$, satisfying $\iota_{\underline{G}}\circ\gamma=\gamma=\iota_{\underline{U}}\circ\gamma$ so
\begin{align*}
F(1,\cdot)\circ\gamma=F(1,\cdot)\circ\iota_{\underline{G}}\circ\gamma=\iota_{\underline{U}}\circ\gamma=\gamma
\end{align*}
Therefore, the homology isomorphism induced by $F(1,\cdot)$ fixes the cycles supported in $X_{\underline{G}}(\mathbb{C})\cap X_{\underline{U}}(\mathbb{C})$.
\end{proof}
\begin{cor}\label{CI-diffeo-type} The diffeomorphism type of smooth projective complete intersection of codimension $k$ is uniquely determined by the degree $(d_1,\cdots,d_k)$ of defining homogeneous polynomials.
\end{cor}
\begin{proof}
It is immediate from Proposition \ref{CI-ambient-isotopy}.
\end{proof}

\end{document}